\documentclass[11pt]{amsart} 
\usepackage{graphicx,stackengine,scalerel}

\usepackage{verbatim, latexsym, amssymb, amsmath,color}
\usepackage{epsfig}
\usepackage{CJK}
\usepackage{amsmath}
\usepackage{amssymb}
\usepackage{amsthm}
\usepackage[utf8]{inputenc}
\usepackage[english]{babel}
\usepackage{xpatch}
\usepackage{systeme}

\usepackage{amsmath,url}
\usepackage{enumitem}

\def\z{\mathbin{\zeta}}
\def\a{\mathbin{\alpha}}
\def\b{\mathbin{\beta}}
\def\w{\mathbin{\omega}}
\def\r{\mathbin{\gamma}}
\def\e{\mathbin{\varepsilon}}
\def\s{\mathbin{\varsigma}}
\def\p{\mathbin{\rho}}
\def\l{\mathbin{\partial}}
\def\n{\mathbin{\nabla}}
\def\G{\mathbin{\Gamma}}
\def\B{\mathbin{\mathcal B}}
\def\T{\mathbin{\mathcal T}}
\def\S{\mathbin{\mathcal S}}
\def\Q{\mathbin{\mathcal Q}}
\def\F{\mathbin{\mathcal F}}
\def\M{\mathbin{\mathcal M}}
\def\O{\mathbin{\Omega}}
\def\H{\mathbin{\mathbb H}}

\def\R{\mathbin{\mathbb R}}
\def\P{\mathbin{\mathcal P}}
\def\C{\mathbin{\mathbb C}}
\def\W{\mathbin{\mathcal W}}
\def\Z{\mathbin{\Sigma}}
\def\N{\mathbin{\mathbb N}}
\def\L{\mathbin{\mathcal L}}

\newtheorem{theorem}{Theorem}[section]
\newtheorem{corollary}[theorem]{Corollary}
\newtheorem{proposition}[theorem]{Proposition}

\newtheorem{lemma}[theorem]{Lemma}
\newtheorem{claim}[theorem]{Claim}
\newtheorem{definition}[theorem]{Definition}
\newtheorem{remark}[theorem]{Remark}

\usepackage{geometry}
\title{Min-max free boundary minimal surface of genus $g\geq 1$}
\author{Yuchin Sun}
\begin{document}
	\maketitle 
\begin{abstract}
In this paper, we build up a min-max theory for minimal surfaces using sweepouts of surfaces of genus $g\geq 1$ and $m\geq 1$ ideal boundary components. We show that the width for the area functional can be achieved by a bubble tree limit consisting of branched genus $g$ free boundary minimal surfaces with nodes, and possibly finitely many branched minimal spheres and free boundary minimal disks. Our result extends the min-max theory by \cite{CD}\cite{LAX}\cite{PR} to all genera. 
\end{abstract}
\section{Introduction}
Existence theory of minimal surfaces is a fundamental object in Riemannian geometry. There are lots of interesting results concerning general existence theory using conformal harmonic parametrization. Sacks-Uhlenbeck developed a general existence theory for minimal surfaces in compact manifold using Morse theory for perturbed energy functional \cite{SU}. Micallef-Moore used Sacks-Uhlenbeck's theory to prove the topological sphere theorem \cite{MD}. Chen-Tian gave a general existence theorem for minimal surfaces of arbitrary genus by extending \cite{SU} to stratified Riemann surfaces. For free boundary minimal surfaces, Fraser \cite{AF} developed a partial Morse theory for finding minimal disk with free boundary in any co-dimensions using the perturbed energy functional approach by Sacks-Uhlenbeck, and proved the existence of solutions with bounded Morse index when the relative homotopy group is nontrivial. 

Colding and Minicozzi used min-max theory to construct min-max minimal spheres and proved the finite time extinction for three-dimensional Ricci flow under certain topological conditions \cite{CD}. Inspired by the work of Douglas and Rad\'{o} they replaced the area by the Dirichlet energy functionaly, and inspired by the Birkhoff's so-called curve-shortening procedure, gave a hramonic replacement procedure. Generalizations for tori and higher genus surfaces was performed by Zhou in \cite{minimaltorus}\cite{minmaxgenus}. Following the work by Colding and Minicozzi, the free boundary minimal disk was proved in \cite{LAX} by Zhou, Lin and Ao, and in \cite{PR} by Laurain and Petrides. In the free boundary case, in order to perform the harmonic replacement procedure, one needs to get the uniqueness of free boundary harmonic map with small energy with respect to the partial Dirichlet boundary data. It's much harder in the general case for free boundary harmonic map and was solved by Zhou, Lin and Ao in \cite{LAX}, and by Laurain and Petrides in \cite{PR}. In this paper, we extend the results of \cite{LAX} and \cite{PR} to all genera free boundary surfaces. Now we state the main result.

\subsection{Main Result}
Given a closed Riemannian manifold $M$ and an embedded submanifold $N$ of co-dimension at least one. Let $\S_0$ be a bordered Riemann surface of genus $g\geq 1$ and with $m\geq 1$ ideal boundary components, and none of the ideal boundary component is a puncture. We consider the following variational space:
\begin{equation*}
  \O = \left \{\r:\S_{0}\times[0,1]\to M,
  \begin{aligned}
     & \r(\cdot,t):[0,1]\to C^0(\bar{\S_{0}},M)\cap W^{1,2}(\S_{0},M)\text{ is continuous, }\\
     & \text{and } \r(\cdot,t)(\partial\S_{0})\subset N,\:\forall t \in[0,1], \\
     & \text{and }\r(\cdot,0), \r(\cdot,1)\text{ are constant maps.}
  \end{aligned} \right\},
\end{equation*}
and we call $\r\in\O$ a \textit{sweepout}. Denote by $\O_{\b}$ the homotopy class of $\b$ in $\O$. The width corresponding to a given homotopy class $\O_{\b}$ is defined to be \[\W(\O_{\b}):=\inf_{\r\in\O_{\b}}\max_{t\in[0,1]}\text{Area}(\r(t)).\]
A smooth bordered Riemann $\S_0$ is said to be of type $(g,m)$ if $\S_0$ is topologically a sphere attached with $g$ handles and $m$ disks removed. It is topologically equivalent to a compact surface of genus $g$ with $m$ punctures. A nodal bordered Riemann surface is of type $(g,m)$ if it is a degeneration of a smooth bordered Riemann surface of the same type.

\begin{theorem}\label{main}
For any homotopically nontrivial sweepout $\r\in\O$. If $\W(\O_{\r})>0$, then there exist a bordered Riemann surface of same type $(g,m)$ (possibly a nodal bordered Riemann surface), conformal harmonic map $u_0:\S\to M$ with free boundary $u_0|_{\l\S}\subset N$, finitely many harmonic spheres $v_i:S^2\to M$, finitely many free boundary minimal disk $\w_j:D\to N$ with $\w_j|_{\l D}\subset N$, and a minimizing sequence $\r_n$ in $\O_{\r}$, $\{t_n\}\subset[0,1]$, such that $\r_n(t_n)$ converges to $\{u_0,v_i,\w_j\}$ in the bubble tree sense, and
\begin{equation}\label{1}
\text{Area}(u_0)+\sum_{i,j}\text{Area}(v_i)+Area(\w_j)=\W(\O_{\r}).    
\end{equation}
\end{theorem}
The novelty of Theorem \ref{main} lies on the area identity. Since the width w.r.t. energy functional equals to the width w.r.t. the area functional (see Theorem \ref{width identity}), and the $u_0$ in Theorem \ref{main} is a conformal free boundary harmonic map, the area identity in \eqref{1} implies the energy identity. The possible loss of energy during the bubble tree convergence has attracted a lot of interests. The energy identity, equivalent to no loss of energy, has played an important role in the study of geometric analysis \cite{TP}\cite{CT}\cite{JJ}\cite{TQing}. In Theorem \ref{main}, the minimizing sequence $\r_n(t_n)$ is actually defined on domain with varying conformal structure. The energy identity of such sequence doesn't hold trure for general case, see \cite{TP} and \cite{Meow}. 

The organization of the paper is the following: in section 2 we review various definitions and properties of the Teichm\"{u}ller space of bordered Riemann Surfaces, in section 3 we define what's the homotopy equivalence for sweepout defined on different domain, in section 4 we prove that the width w.r.t. area functional is equivalent to the width w.r.t. energy functional, and we can perturb the minimizing sequence so that the energy and area of it is sufficiently close, in section 5 we build a systematic replacement procedure for free boundary harmonic replacement, in section 6 we state the convergence result, discuss the case where the domain surface degenerates and prove Theorem  \ref{main}.
\subsection*{Acknowledgments}
The author would like to express her gratitude to Xin Zhou and Longzhi Lin for all the helpful guidance and constant encouragement.
\section{Teichm\"{u}ller space of bordered Riemann Surfaces}

The main reference is \cite{Abikoff} and \cite{teichmuller}. All the materials presented in the section trace back to the books.
\subsection{Teichm\"{u}ler space of closed Riemann surfaces}
\begin{description}

\item[Marked Surface and Fuchsian Group] Given a fixed closed Riemann surface $\Z_0$ of genus $g\geq 2$, consider all the surfaces $(\Z,f)$, where. $f:\Z_0\to\Z$ is an orientation-preserving diffeomorphism. We say that $(\Z,f)$ and $(\Z',g)$ are equivalent if $g\circ f^{-1}:\Z\to\Z'$ is homotopic to a conformal diffeomorphism from $\Z$ to $\Z'$. We call $(\Z,f)$ a \textit{marked surface}. The set of all equivalent classes of marked surfaces $\{[(\Z,f)]\}$ is the Teichm\"{u}ller spaces $\T(\Z_0)$ of $\Z_0$. By the Uniformization Theorem, all the closed surfaces $\Z_0$ with genus $g\geq 2$ have their universal covering space the upper half plane $\H$. The covering transformatin group of $\pi:\H\to\Z_0$ is called \textit{Fuchsian group}, which will be denoted by $\G$ and $(\Z_{0},\G)$ is called \textit{Fuchsian model}. The holomorphic diffeomorphism group of $\H$ is PSL$(2,\R)$, so $\G\subset\text{PSL}(2,\R)$. Given a Fuchsian model $(\Z,\G)$, there's a set of normalized generators $\{\a_{i},\b_i\}_{i=1}^{g}$ for $\G$, where $\a_{g}$ has attractive fixed point at $1$ and $\b_{g}$ has repelling and attractive fixed point at $0$ and $\infty$ respectively. Moreover, the set of normalized generators is uniquely determined by the equivalent class in $\T(\Z_{0})$. $\{\a_i\}$, $\{\b_i\}$ can be uniquely written as $\a_i=\frac{a_iz+b_i}{c_iz+d_i}\in\text{PSL}(2,\R)$ and $\b_i=\frac{a_i'z+b_i'}{c_i'z+d_i'}\in\text{PSL}(2,\R)$. Hence we can define the \textit{Fricke coordinates}:  
$\F_{g}:\T(\Z_0)\to\R^{6g-6}$ as the following
\[\F_g([(\Z,f)])=(a_i,c_i,d_i,a_i',c_i',d_i')_{i=1}^{g-1}.\]
$\F_g$ is injective by \cite[Theorem 2.25]{teichmuller}. Hence we have an induced topology on $\T(\Z)$ by the Fricke coordinates.  

\item[Quasi-conformal Map]
We introduce quasi-conformal maps and combine the marked surface model with the quasi-conformal maps. Let $f:\Z_{0}\to\Z$ be an orientation-preserving diffeomorphism between two Riemann surfaces. Given local complex coordinates $(z,\bar{z})$, $(w,\bar{w})$ on $\Z_0$ and $\Z$. Denote $f(z)=w\circ f\circ z$. The Beltrami coefficient $\mu_f$ is defined by $\frac{f_{\bar{z}}}{f_z}$. If $|\mu_f|\leq k<1$, then we call such $f$ a quasi-conformal map. Given a marked surface $(\Z,f)$, $f:\Z_{0}\to\Z$. By \cite[Theorem 5.8]{teichmuller}, there exists a unique holomorphic quadratic differential $\phi$ on $\Z_0$ with $L^1$-norm $\|\phi\|$ less than $1$, and a unique quasi-conformal map $f_\phi :\Z_{0}\to\Z$ homotopic to $f$, such that $\mu_{f_\phi}=\|\phi\|\frac{\bar{\phi}}{|\phi|}$. Denote the set of all holomorphic quadratic differentials on $\Z_0$ with $L^1$-norm less than one by $\Q^1_{\Z_0}$. 
\begin{theorem}\cite[Theorem 5.15]{teichmuller}
$\mathcal{F}:\Q_{\Z_{0}}^{1}\to\T(\Z_{0})$ defined by $\mathcal{F}(\phi)=[(\Z ,f_\phi )]$ is a homeomorphism.
\end{theorem}
\begin{remark}\label{homeo}
Since the map $\F$ is a homeomorphism and the equivalent class of each marked surface uniquely determines a set of normalized generators, then the map defined by 
\[\Phi:\Q^1_{\Z_0}\to\F_g(\T(\Z_0)),\quad\Phi(\phi)=\F_g([(\Z,f_\phi)])\]
is also a homeomorphism.
\end{remark}
 So we know that for a given $\phi\in\Q^{1}_{\Z_{0}}$, there exists the corresponding equivalent class of a marked surface $\mathcal{F}(\phi)=[(\Z,f_{\phi})]$ and the corresponding normalized Fuchsian model $(\Z_{0},\G).$
\end{description}

\subsection{Teichm\"{u}ler space of bordered Riemann surfaces}\label{bordered riemann surface}

Now we consider Riemann surface with boundary. Let $\S$ be a bordered Riemann surface of genus $g$ and $m$ ideal boundary components, and none of the ideal boundary component is a puncture. Let $\pi:\H\to\S$ be the universal cover map. Given an ideal boundary component $\z$, we can take a simple closed curve $C$ surrounding $\z$ and the homotopy class $[C]$ determine a deck transformation $\b$ in $\G$ which is the deck transformation group of $\pi:\H\to\S$. The neighborhood of $\z$ is conformal to an annulus and $\b$ is hyperbolic. We can lift $C$ to a $\b$ invariant curve $\Tilde{C}$ joining two fixed points of $\b$. Since $C$ is separating $\S$, $\Tilde{C}$ is separating $\H$. One of the components $A$ of $H\setminus\Tilde{C}$ projects to a neighborhood of $\z$ which is an annulus. The boundary $\partial A$ of $A$ lies on $\hat{\R}$. If we add those boundary components to $A$ then they form the bordered component in the projection under $\pi$. $\S$ can be viewed as surface with boundary whose interior $\S^0$ has the structure of a Riemann surface. We call $\S\setminus\S^0$ the \textit{border} of $\S$.

We define the double $\S^d$ of $\S$ to be $\S\cup\bar{\S}$ with corresponding points one the border identified, where $\bar{\S}$ is the mirror image of $\S$. More precisely, To each $x\in\S$, assign a point $\bar{x}=\sigma(x)$. If $x\in\S^0$, let $\S^0\supset N\xrightarrow{z_{\a}}\C$ be a chart at $x$. As a chart at $\bar{x}$ take $\bar{N}=\{\bar{x}|\:\sigma^{-1}(\bar{x})\in N\}$ with local coordinate $\bar{z}_{\a}=J\circ z_{\a}\circ \sigma^{-1}$ where $J:\C\to\C$ is the complex conjugation. For $x\in\partial\S$, we use the obvious topology. This defines the mirror image bordered surface $\bar{\S}$ of $\S$. Then $\S^d=\S\cup_\sigma\bar{\S}$.

Now we focus on the case $g\geq 1$ and $m\geq 1$ so that the double $\S^d$ carries a hyperbolic structure and $i$ is covered by an isometry of the universal cover of $\S^d$. 
\begin{claim}\label{border geodesic}
Each border component $\partial\S$ of $\S$ is the geodesic in its free homotopy class on $\S^d$. 
\end{claim}
\begin{definition}\label{admissible diff}
We call $\w$ an admissible quadratic differential on $\S$ if $\w$ is the restriction to $\S$ of a holomorphic quadratic differential $\w^d$ on the doubled surface $\S^d$ such that $\int_{\partial\S}\sqrt{\w}\in\R$. We denote the space of all the admissible quadratic differentials on $\S$ by $\Q_{\S}$.
\end{definition}
Let $\w$ be an admissible quadratic differential on $\S$. Away from zeros of $\w$ we can define a local coordinate by
\[\xi+i\eta=\int\sqrt{w}.\]
If we let $\xi'+i\eta'=K\xi+i\eta$ be another local coordinates, where $K>1$ is a constant, we can define a new conformal structure on $\S$ away from zeros. Such conformal structure extends to the zeros so we have a new conformal structure $(\S,\w,k)$ with $k=\frac{K-1}{K+1}$. We call this new conformal structure \textit{Teichm\"{u}ller deformation} of $\S$ given by $\w$ and $k$. As above, $\Q_{\S}$ denote the space of admissible quadratic differentials on $\S$. Using the norm $\|w\|=\int_{\S}|w|$, we denote $\Q_{\S}^1$ to be the open unit ball of $\Q_{\S}$.
\begin{definition}\label{teichmuller space}
The Teichm\"{u}ller space $\T(\S)$ on a bordered Riemann surface $\S$ is defined as the set of all the Teichm\"{u}ller deformation $(\S,\frac{\w}{\|\w\|},\|w\|)$ by elements $\w\in\Q_{\S}^1$ with the topology induced by that on $\Q^1_{\S}$. 
\end{definition}

As we mentioned before, there are various definitions for the Teichm\"{u}ller space of closed Riemann surface, now we compare those definitions with Definition \ref{teichmuller space}. Given $\w\in\Q_{\S}^1$, by Definition \ref{admissible diff} $w$ is the restriction to $\S$ of a holomorphic quadratic differential $\w^d$ on its double $\S^d$. By the previous discussion of Teichm\"{u}ler space for closed Riemann surface with genus $g>1$, we can normalize $\w^d$ so that the $L^1$ norm of it is strictly less than one (we still denote the normalized holomorphic quadratic differential by $\w^d$). There exists a unique quasi-conformal map $f_{\w^d}:\S^{d}\to\bar{\S}^{d}$, where $\bar{\S}^{d}$ is another closed Riemann surface with same genus as $\S^d$, such that its Beltrami coefficient is equivalent to $\|\w^{d}\|\frac{\bar{\w}^d}{|\w^{d}|}$. The corresponding marked surface $\mathcal{F}(\phi)=[(\bar{\S}^{d},f_{\w^d})]$ has the same conformal structure as the Teichm\"{u}ller deformation  $(\S^{d},\frac{\w^{d}}{\|\w^{d}\|},\|\w^{d}\|)$. We can view the restriction of $f_{\w^d}$ on $\S$ as the corresponding quasi-conformal map for $\w\in\Q_{\S}^1$. 

We can also choose the corresponding Fuchsian group for $\S$ from the Fuchsian group of its double $\S^d$. By the previous discussion, there is a unique normalized Fuchsian group $\G$ such that $\S^d=\H/\G$. Now we lift the quasi-conformal map $f_{\w^d}$ up to $\H$ by the covering maps $\pi:\H\to\S^d$ and $\bar{\pi}=\H\to\bar{\S}^d$ to get $\tilde{f}_{\w^d}:\H\to\H$. The fundamental group $\pi_1(\S)$ injects into $\G$ and its image $\G'$ stabilizes a simply connected component $\H'$ of $\H\setminus\pi^{-1}(\partial\S)$. Then we get $\S=\H'/\G'$. By the above discussion we also get the conformal structure of the following Teichm\"{u}ller deformation $(\S,\frac{\w}{\|\w\|},\|\w\|)$ is equivalent to $\tilde{f}_{\w^d}(\H')/\tilde{f}_{\w^d}\circ\G'\circ\tilde{f}_{\w^d}^{-1}$.
\section{Homotopy equivalence of the variational space}
Given a closed Riemannian manifold $(M,g)$, isometrically embedded in $\R^N$, and an embedded submanifold $N$ of co-dimension at least one. Let $\S_0$ be a fixed Riemann surface of genus $g\geq 1$ with $m\geq 1$ ideal boundary components, none of the ideal boundary components is a puncture, so that the double $\S_{0}^d$ carries a hyperbolic structure. Let $\s:[0,1]\to\T(\S_{0})$ denote the continuous map to $\T(\S_{0})$ with respect to the $L^1$ topology of $\Q_{\S_{0}}^1$. From Definition \ref{teichmuller space} we know that for each $\s(t)\in\T(\S_{0})$, there is a corresponding $\w(t)\in\Q_{\S_{0}}^1$ such that $\s(t)=(\S_{0},\frac{\w(t)}{\|\w(t)\|},\|\w(t)\|)$. Since $\w(t)$ is the restriction to $\S_{0}$ of a holomorphic quadratic $\w^{d}(t)$ on its double $\S^{d}$, we can normalize $\w^d$ so that $\w^d\in\Q^{1}_{\S_{0}^d}$. There exists a unique quasi-conformal map $f_{\s(t)}$ such that its Beltrami coefficient is equivalent to $\|\w^{d}\|\frac{\bar{\w^d}}{|\w^d|}$. Moreover, $f_{\s(t)}(\S_{0}^d)$ has the same conformal structure as $(\S^{d}_{0},\frac{\w^{d}(t)}{\|\w^{d}(t)\|},\|\w^{d}(t)\|)$, and $f_{\s(t)}(\S_{0})$ has the same conformal structure as $\S_{\s(t)}=(\S_{0},\frac{\w(t)}{\|\w(t)\|},\|\w(t)\|)$. The Teichm\"{u}ller deformation $\s(t)=(\S_{0},\frac{\w(t)}{\|\w(t)\|},\|\w(t)\|)$ also corresponds to the image of quasi-conformal map $f_{\s(t)}$ restricted on $\S_0$, i.e., $f_{\s(t)}(\S_{0})$. We denote $f_{\s(t)}(\S_{0})$ by $\S_{\s(t)}$. 
\begin{definition}\label{variational space}
The variational spaces are defined as 
\begin{equation*}
  \O = \left \{\r:\S_{0}\times[0,1]\to M,
  \begin{aligned}
     & \r(\cdot,t):[0,1]\to C^0(\bar{\S_{0}},M)\cap W^{1,2}(\S_{0},M)\text{ is continuous, }\\
     & \text{and } \r(\cdot,t)(\partial\S_{0})\subset N,\:\forall t \in[0,1], \\
     & \text{and }\r(\cdot,0), \r(\cdot,1)\text{ are constant maps.}
  \end{aligned} \right\},
\end{equation*} 
and
\begin{equation*}
  \Tilde{\O} = \left \{(\r(t),\s(t)),\:
  \begin{aligned}
     & \r(\cdot, t):[0,1]\to C^0(\bar{\S}_{\s(t)}(t),M)\cap W^{1,2}(\S_{\s(t)},M)\text{ is continuous, }\\
     & \text{and }
     \s\in C^0([0,1],\T(\S_{0})),\\
     & \text{and } \r(\cdot,t)(\partial\S_{\s(t)})\subset N,\:\forall t \in[0,1], \\
     & \text{and }\r(\cdot,0), \r(\cdot,1)\text{ are constant maps.}
  \end{aligned} \right\}.
\end{equation*} 
Given a map $\r\in\O$, we define $\O_{\r}$ to be the homotopy class of $\r$ in $\O$. Similarly, for $(\r,\s)\in\tilde{\O}$, we denote $\tilde{\O}_{(\r,\s)}$ to be the homotopy class of $(\r,\s)$ in $\tilde{\O}$. Each $\r\in\O$ will be called a sweepout.
\end{definition}
For a given $(\r,\s)\in\tilde{\O}$, the continuity of $\r(\cdot, t)\in C^0\big([0,1],C^0(\bar{\S}_{\s(t)},M)\cap W^{1,2}(\S_{\s(t)},M)\big)$ with respect to $t$ is defined as the following: for each $\S_{\s(t)}$ there exists a corresponding quasi-conformal map $f_{\s(t)}$. We can pull the path $\r(\cdot,t):\S_{\s(t)}\to M$ back to $\S_0$, i.e., \[f^*_{\s(t)}(\r(\cdot,t))=\r(\cdot,t)\circ f_{\s(t)}:\S_{0}\to M.\]
The continuity of $\r(\cdot,t)$ with respect to $t$ is defined as the continuity of the path $f^*_{\s(t)}(\r(\cdot,t))$ with respect to $t$ on the same domain $\S_0$. 

On the other hand, we define what is the homotopy equivalence in $\tilde{\O}$ with different domain $\S_{\s(t)}$. Given two elements $\{(\r_i(t),\s_i(t)):\:i=1,2\}$. As above, we use the quasi-conformal map $f_{\s_{i}(t)}:\S_{0}\to\S_{\s_i(t)}$, $i=1,2$ to pull the path $\r_i(\cdot,t):\S_{\s_i(t)}\to M$ back to $\S_0$, i.e., $f^*_{\s_i(t)}\r_i(\cdot,t)=\r_i(\cdot,t)\circ f_{\s_i(t)}:\S_{0}\to M$. We say that
$\{(\r_1(t),\s_1(t))\}$ is homotopic to $\{(\r_2(t),\s_2(t))\}$ if $f^*_{\s_1(t)}\r_1(\cdot,t)$ is homotopic to $f^*_{\s_2(t)}\r_2(\cdot,t)$.
\begin{definition}[Width]
We denote $E(\cdot)$ and $\text{Area}(\cdot)$ as the Dirichlet energy and area functional. Associated to each homotopy class $\O_{\r}$, we define the width of area functional to be the following:
\[\W(\O_{\r}):=\inf_{\p\in\O_{\r}}\max_{t\in[0,1]}\text{Area}(\p(t)).\]
Similarly, for a homotopy class $\tilde{\O}_{(\r,\s)}$, the corresponding width for energy functional is defined to be the following:
\[\W_{E}(\tilde{\O}_{(\r,\s)}):=\inf_{(\p,\s)\in{\tilde\O_{(\r,
\s)}}}\max_{t\in[0,1]}E(\p(t),\S_{\s(t)}).\]
\end{definition}
\begin{definition}
Given a sweepout $\tilde{\r}\in\O$, it has the corresponding homotopy class $\O_{\tilde{r}}$. We called a sequence of sweepouts $\{\r_j\}_{j\in\N}\subset\O_{\tilde{\r}}$ a minimizing sequence if
\[\lim_{j\to\infty}\max_{t\in[0,1]}\text{Area}(\r_j(t))=\W(\O_{\tilde{\r}}).\]
\end{definition}
\section{Conformal Parametrization}
\begin{theorem}\label{width identity}
Given a sweepout $\tilde{\r}\in\O$, we have
\[\W_{E}(\tilde{\O}_{(\tilde{\r},\s_0)})=\W(\O_{\tilde{\r}}),\]
where $\s_0$ is the constant map $\s_0(t)=\S_0$ for all $t\in[0,1]$.
\end{theorem}
\begin{proof}
For a given sweepout $\tilde{\r}\in\Omega$ we always have that 
\[\text{Area}(\tilde{\r}(t))\leq E(\tilde{\r}(t),\S_0),\quad\forall t \in [0,1],\]
so we get $\W(\O_{\tilde{\r}})\leq\W_{E}(\tilde{\O}_{(\tilde{\r},\s_0)})$
It suffices to show that there exists a sequence of sweepouts $\{(\r_j,\s_j)\}_{j\in\N}\subset\tilde{\O}_{(\tilde{\r},\s_0)}$ (or equivalently, $\{f^*_{\s_j}\r_j\}_{j\in\N}\subset\Omega_{\tilde{\r}}$) such that
\begin{equation}\label{EA1}
\lim_{j\to\infty}\max_{t\in[0,1]}\text{Area}(\r_j(t))=\W(\O_{\tilde{\r}}),    
\end{equation}
and 
\begin{equation}\label{EA2}
\lim_{j\to\infty}\text{Area}(\r_j(t))-E(\r_j(t),\s_j(t))=0.    
\end{equation}
By definition $\W_{E}(\tilde{\O}_{(\tilde{\r},\s)})\leq \max_{t\in[0,1]}E(\r_j(t),\s_j(t))$, \eqref{EA1} and \eqref{EA2} imply that as $j\to\infty$ we have \[\W_E(\tilde{\O}_{(\tilde{\r},\s_0)})\leq\max_{t\in[0,1]}\text{Area}(\r_j(t))=\W(\O_{\tilde{\r}}).\]

Now we show the existence of the sequence of sweepouts which satisfies \eqref{EA1} and \eqref{EA2}.
\begin{claim}\label{regularization}
For a given sweepout $\tilde{\r}$, we can find a regularized sweepout in the same homotopy class $\r'\in\O_{\tilde{\r}}$, which lies in $C^0([0,1],C^2(\S_0,M))$ as a map of $t$, such that $\r'$ is close to $\tilde{\r}$ uniformly in the $W^{1,2}\cap C^0$-norm for all $t\in[0,1]$.
\end{claim}
\begin{proof}
It follows from a standard argument using mollification just like \cite[Lemma D.1]{CD}. We point out necessary modification of \cite[Lemma D.1]{CD} to make sure that the image of $\partial \S_0$ under each slice $\r'(\cdot,t)$ lie in the constraint submanifold $N$. In particular, near the boundary $\partial\S_0$ we can first enlarge the domain $\S_0$ in its double $\S^{d}_0$ by reflecting $\tilde{\r}(\cdot,t)$ across $N$ in the Fermi coordinates around $N$. The mollified maps when restricted to $\S_0$ will map $\partial\S_0$ to $N$ by our construction, and it's closed to $\tilde{\r}(\cdot,t)$ uniformly in $W^{1,2}\cap C^0$ near $\partial\S_0$. In the interior of $\S_0$, we can mollify $\tilde{\r}(\cdot,t)$ in $\R^n$. To combine them, we can choose a partition of unity to glue these two mollifications together. Finally, we can follow \cite[Lemma D.1]{CD} to project the mollified map back to $M$ by nearest point projection map to get the desired $\r'(\cdot,t)$. By choosing the mollification parameter small enough we can make sure that $\max_{t\in[0,1]}\|\tilde{\r}(\cdot,t)-\r'(\cdot,t)\|$ is as small as we want, where the norm is $W^{1,2}\cap C^0(\S_0,\R^n)$. Note that an explicit homotopy between $\r'$ and $\tilde{\r}$ is given by letting the mollification parameter goes to $0$. So we finish the proof of the claim. 
\end{proof}

\begin{claim}\label{uniformization}
If the regularized sweepout $\r'$ given in Claim \ref{regularization} induced non-degenerate metrics $\r'(t)^*g$ on $\S_0$,  $\forall t\in[0,1]$. Then there exists a continuous map $\s:[0,1]\to\T(\S_0)$ and orientation-preserving $C^{1,\frac{1}{2}}$ conformal diffeomorphism $h(t):\S_{\s(t)}\to (\S_0,\r'(t)^*g)$ such that the re-parametrized $\r(\cdot,t)=\r'(h(t),t)$ has the following properties:
\begin{enumerate}
    \item $(\r,\s)\in\tilde{\O}_{\tilde{\r},\s_0},$
    \item $\text{Area}(\r)=E(\r,\s).$
\end{enumerate}
\end{claim}
\begin{proof} 
\begin{description}
\item[1]The existence of conformal diffeomorphism $h$.

We denote the double of $\S_0$ by $\S^d_0$ and the covering map $\pi_0:\H\to\S^d_0=\H/\G_0$, where $\G_0$ is the Fuchsian group of $\S^d_0$. Fix some $t\in[0,1]$. The regularized $C^2$ sweepout $\r'(t):\S_0\to (M,g)$ induced a $C^1$ metric on $\S_0$, i.e., the pulled-back metric $\r'(t)^*g$. We extend the pulled-back metric $\r'(t)^*g$ on $\S_0^d$ by the mirror image of $\S_0$, i.e., having an isometric copy of $(\S_0,\r'(t)^*g)$ with the opposite orientation and identifying the boundary points. The mollification done in Claim \ref{regularization} near $\partial\S_0$ ensures that $\r'(t)^*g$ extends to a $C^2$ metric on $\S_0^d$, and we denote this $C^2$ metric by $g'(t)$. Pull $g'$ back to $\H$ by $\pi_0$ and denote it still by $g'$, then it's invariant under the $\G_0$ group action. In the complex coordinates $\{z,\bar{z}\}$, we can write
\begin{equation}\label{pull back metric}
g'(t)=\lambda(t)|dz+\mu(t)d\bar{z}|^2,\quad |\mu(t)|\leq k<1,    
\end{equation}
where $\lambda(t):\H\to\R$, $\lambda(t)>0$ and $\mu(t)$ is the Beltrami coefficient. The nondegenerate assumption of the metric ensures that $|\mu|$ is bounded away from $1$. Then we have a unique normalized quasi-conformal mapping:
\[f^{\mu}:(\H,g')\to(\H,g_0),\]
where $g_0$ is the standard Poincar\'{e} metric. Now push forward the Fuchsian group $\G_0$ under $f^\mu$. Since $f^\mu$ is a homeomorphism, we get another Fuchsian group 
\[\G_{\s}:=f^\mu\circ\G_0\circ (f^\mu)^{-1}.\]
$\G_{\s}$ gives a normalized Fuchsian model which represents an element in $\T(\S^d_0)$. By Definition \ref{admissible diff} and Definition \ref{teichmuller space} we know that each element in $\T(\S_0)$ corresponds to the restriction of an element in $\T(\S_0^d)$. Since we extend $g'$ on $\S^d_0$ by the mirror image of $\S_0$, the image of $\pi_0^{-1}(\partial\S_0)$ under $f^\mu$ is still $\pi_0^{-1}(\partial\S_0)$. This implies that the normalized Fuchsian model $\G_{\s}$ which represents an element in $\T(\S^{d}_0)$ corresponds to an element in $\T(\S_0)$. We denote this element by $\S_{\s}$, and its double $\S_{\s}^d=\H/\G_{\s}$. Let $\pi_{\s}:\H\to\S_{\s}^d$ be the quotient map. Then after taking quotient of $f^\mu$ by $\pi_0$ and $\pi_{\s}$, we get
\[f^\mu :(\S^d_0,g')\to(\S^d_{\s},g_0).\]
$f^\mu$ is conformal between $(\S^d_0,g')$ and $(\S^d_{\s},g_0)$. So $f^\mu$ is conformal between $(\S_0,g')$ and $(\S_{\s},g_0)$. Let \[h(t):(\S_{\s(t)},g_0)\to(\S_0,g'(t)),\quad h:=(f^\mu)^{-1},\]
then for each $t\in[0,1]$, $h(t)$ is a conformal homeomorphism between $(\S_{\s(t)},g_0)$ and $(\S_0,g'(t))$. The $C^{1,\frac{1}{2}}$-regularity follows from \cite[Corollary 3.3.1]{JJ}. For each $t\in[0,1]$, by \eqref{pull back metric} we have the corresponding $\mu(t)$ for $g'(t)$, and the corresponding induced element in $\T(\S_0)$ is  denoted by $S_{\s(t)}\in\T(S_0)$.
\item[2] Continuity of $\s:[0,1]\to\T(\S_0)$ and $\r=\r'(h(t),t)$ with respect to $t$. 

From Claim \ref{regularization} we know that the regularized $\r'(t)\in C^0([0,1],C^2(\S_0,M))$. So the induced pull-back metric $g'(t)$ is continuous with respect to $t$. This implies that the Beltrami coefficient $\mu(t)$ is also continuous w.r.t. $t$. By \cite{AB} and \cite{minimaltorus} we know that the quasi-conformal map $f^{\mu(t)}$ and $h(t)=(f^{\mu(t)})^{-1}$ are both continuous w.r.t. the parameter $t$, and the restriction of $f^{\mu(t)}$ on $\S_0$ is also continuous w.r.t. $t$.

Let us now show the continuity of the map $\s:[0,1]\to \T(\S_0)$ w.r.t. the parameter $t$. Recall that $\G_0$ is the Fushian group of the double $\S^d_0$ Now the corresponding normalized Fuchsian group $\G_{\s(t)}$ for the double $S^d_{\s(t)}$ is given by 
\[\G_{\s(t)}:=f^{\mu(t)}\circ\G_0\circ  h(t).\]
Let $\{\a_i^0,\b_i^0\}_{i=1}^{g}$ be the normalized generators for $\G_0$. So for a fixed $\a_i^0\in\G_0$, and let $\a_i^0(t)=f^{\mu(t)}\circ\a_i^0\circ  h(t)$ and $\b_i^0(t)=f^{\mu(t)}\circ\b_i^0\circ  h(t)$. Then $\a_i^0(t)$ and $\b_i^0(t)$ are both continuous w.r.t. $t$, which means that the coefficient of the linear fractional transformations corresponding to   $\{\a_i^0(t),\b_i^0(t)\}_{i=1}^{g}$ are continuous functions of $t$. Since $\{\a_i^0(t),\b_i^0(t)\}_{i=1}^{g}$ form the normalized generators for $\G_{\s(t)}$. Now using the topology of Fricke space we know that the corresponding elements $\S_{\s(t)}\in\T(\S_0)$ are continuous w.r.t. $t$.
The continuity of $\r(t)=\r'(h(t),t)$ follows from the continuity of $h$ w.r.t. $t$.

\item[3] $(\r,\s)\in\tilde{\O}_{\tilde{\r},\s_0}$
The only thing left to check is that $(\r,\s)$ is homotopic to $(\tilde{\r},\s_0)$. Since for each fixed $t$ we have
\[\r(t):\S_{\s(t)}\to M.\]
For each $\S_{\s(t)}$, there exists a unique quasi-conformal map $f_{\s(t)}:\S_0\to\S_{\s(t)}$ which is the restriction of the quasi-conformal map of its double $f_{\s(t)}:\S^d_0\to\S^d_{\s(t)}$. The normalized Fuchsian group of  $\S^d_{\s(t)}$ is $\G_{\s(t)}:=f_{\s(t)}\circ\G_0\circ (f_{\s(t)})^{-1}$, which coincides with  $\G_{\s(t)}:=f^{\mu(t)}\circ\G_0\circ (f^{\mu(t)})^{-1}$. Thus implies that $f_{\s(t)}$ is homotopic to $f^{\mu(t)}$, and $h(t)$ is homotopic to $(f_{\s(t)})^{-1}$. Therefore we have
\[f^*_{\s(t)}\r(t)=\r(t)\circ f_{\s(t)}=\r'(h(f_{\s(t)}),t).\]
So $\r'(h(f_{\s(t)}),t)$ is homotopic to $\tilde{\r}(t).$
\end{description}
\end{proof}
Let $\{(\tilde{\r}_j,\s_0)\}_{j\in\N}\subset\tilde{\O}_{(\tilde{\r},\s_0)}$ be a sequence of sweepouts such that
\begin{equation}
\lim_{j\to\infty}\max_{t\in[0,1]}\text{Area}(\tilde{\r}_j(t))=\W(\O_{\tilde{\r}}).
\end{equation}
by Claim \ref{regularization} there exists the regularized sequence of sweepouts $\{({\r}'_j,\s_0)\}_{j\in\N}\subset\tilde{\O}_{(\tilde{\r},\s_0)}$ such that $\{(\r'_j,\s_0)\}_{j\in\N}\subset C^0([0,1],C^2(\S_0,M))$ and \[\lim_{j\to\infty}\max_{t\in[0,1]}\text{Area}(\r'_j(t))=\W(\O_{\tilde{\r}}).\]
We consider the pulled-back metric $\tilde{g}_j(t)=\r_j'(t)^*g$, which extends to a $C^2$ metric on the double $\S_0^d$, and it's continuous w.r.t. $t$. Since $\tilde{g}_j(t)$ may be degenerate, let
\begin{equation}\label{pullback metric}
g_j(t):=\tilde{g}_j(t)+\delta_j g_0,    
\end{equation}
where $g_0$ is the standard Poincar\'{e} metric on $\S_0^d$ restricted on $\S_0$, and $\{\delta_j\}$ is a sequence such that $\delta_j\to 0$ as $j\to\infty$. Then by Claim \ref{uniformization}, $g_j(t)$ uniquely determined $\s_j(t)\subset\T(\S_0)$ and conformal diffeomorphism \[h_j(t):(\S_{\s_j(t)},g_0)\to(\S_0,g_j(t)).\]
Now we define the sequence of sweepouts $\{\r_j\}_{j\in\N}$, $\r_j:(\S_{\s_j(t)},g_0)\times[0,1]\to(M,g)$ to be the following
\[\r_j(t)=\r_j'(h_j(t),t).\]
$\r_j$ is homotopic to $\tilde{\r}_j$ by Claim \ref{regularization} and Claim \ref{uniformization}. Since reparametrization doesn't change area so we have 
\[\lim_{j\to\infty}\max_{t\in[0,1]}\text{Area}(\r_j(t))=\W(\O_{\tilde{\r}}).\]
We are left to check that $\{\r_j\}_{j\in\N}$ satisfies 
\begin{equation}
\lim_{j\to\infty}\text{Area}(\r_j(t))-E(\r_j(t),\s_j(t))=0.    
\end{equation}
The following estimate is similar to \cite[Appendix D]{CD} and \cite[Theorem 3.1]{minimaltorus},
\begin{align*}
    E(\r_j(t),\s_j(t))&=E(h_j:\S_{\s(t)}\to(\S_0,\tilde{g}_j(t)))\\
    &\leq E(h_j:\S_{\s(t)}\to(\S_0,g_j(t)))\\
    &=\text{Area}(h_j:\S_{\s(t)}\to(\S_0,g_j(t)))\\
    &=\text{Area}(\S_0,g_j(t))=\int_{\S_0}[\text{det}(g_j(t))]^{1/2}d\text{vol}_0\\
    &\leq\int_{\S_0}[\text{det}(\tilde{g}_j(t))+\delta_j\text{Tr}_{g_0}\tilde{g}_j(t)+C(\tilde{g}_j(t))\delta_j^2]^{1/2}d\text{vol}_0\\
    &\leq\text{Area}(\r_j(t))+C_1(\tilde{g}_j(t))\sqrt{\delta_j}.
\end{align*}
Since $\delta_j\to 0$. We have 
\[\lim_{j\to\infty}\text{Area}(\r_j(t))-E(\r_j(t),\s_j(t))=0.\]
\end{proof}
\begin{corollary}\label{almost harmonic}
Given a sweepout $\tilde{\r}\in\Omega$, there exists a minimizing sequence of reparameterized sweepouts  $\{(\r_j,\s_j)\}_{j\in\N}\subset\tilde{\Omega}_{(\tilde{\r},\s_0)}$ such that
\begin{equation}\label{almost conformal}
\lim_{j\to\infty}\text{Area}(\r_j(t))-E(\r_j(t),\s_j(t))=0.
\end{equation}
\end{corollary}
\section{Replacement procedure and Energy decreasing map}\label{replacement}
In this section, we aim at building a systematic replacement procedure in order to obtain a sequence which satisfies a kind of Palais-Smale assumption. Most of the results presented in the section is similar to \cite[section 3.3]{CD} under different assumption, similar results also can be found in \cite{minimaltorus}\cite{minmaxgenus}\cite{PR}\cite{LAX}. We include the proof and the construction here for the completeness of the paper.

Given a sweepout $\tilde{\r}$, let $\{(\r_j,\s_j)\}_{j\in\N}$ be a minimizing sequence which satisfies \eqref{almost conformal} obtained by Corollary \ref{almost harmonic}. We have
\[\r_j(t):(\S_{\s_j(t)},g_0)\to(M,g),\]
where $g_0$ is the standard Poincar\'{e} metric on $\S_{\s_j(t)}$. Recall that for each $\s_j(t)$ there exists a normalized Fuchsian group $\G_{\s_j(t)}$ for the double $\S^d_{\s_j(t)}$, and the corresponding quotient map \[\pi_{\s_j(t)}:\H\to\H/\G_{\s_j(t)}=\S^d_{\s_j(t)}.\]
We can also view $\r_j(t)$ as being lifted up to a proper subset of $\H$ by $\pi^{-1}_{\s_j(t)}(\S_{\s_j(t)})$. Now we claim that although the quotient map $\pi_{\s_j(t)}$ is different for each $j$ and $t$, we can lift any minimizing sequence $\{(\r_j,\s_j)\}_{j\in\N}$ obtained by Corollary \ref{almost harmonic} up by the different quotient map $\pi_{\s_j(t)}$ to the same subset in $\H$. 
\begin{claim}\label{invariant of boundary}
$\pi^{-1}_{\s_j(t)}(\partial\S_{\s_j(t)})$ is fixed for all $t\in[0,1]$, and all $j\in\N$. 
\end{claim}
\begin{proof}
Since the quotient map $\pi_{\s_j(t)}:\H\to\H/\G_{\s_j(t)}=\S^d_{\s_j(t)}$ is defined w.r.t. the Fuchsian group $\G_{\s_j(t)}$. We recall from Claim \ref{uniformization} that the normalized Fuchsian group $\G_{\s_j(t)}$ is obtained from the quasi-conforaml map $f^{\mu_j(t)}:\S^d_0\to\S^d_{\s_j(t)}$, where $\mu_j(t)$ is the Beltrami coefficient of the non-degenerate $C^2$ pulled back metric defined on $\S_0^d$ (see \eqref{pull back metric}). Since the construction of double surface is done by mirror image and identifying boundaries (see \ref{bordered riemann surface}). We know that $f^{\mu_j(t)}(\partial\S_0)$ is fixed for all $j$ and $t$. Thus the corresponding quotient map $\pi^{-1}_{\s_j(t)}(\partial\S_{\s_j(t)})$ is fixed for all $t\in[0,1]$, and all $j\in\N$.
\end{proof}

Since $\pi^{-1}_{\s_j(t)}(\partial\S_{\s_j(t)})$ is fixed for all $t\in[0,1]$, and all $j\in\N$. We can fix a simply connected component  $\H'$ of $\H\setminus\pi^{-1}_{\s(t)}(\partial\S_{\s(t)})$ and view $\r_j(t)$ as being lifted up to $\H'$ by $\pi^{-1}_{\s(t)}(\S_{\s_j(t)})$. We denote the lifted mappings again by $\r_j(t)$ in abuse of notation. Moreover, the fundamental group  $\pi_1(\S_{\s_j(t)})$ injects into $\G_{\s_j(t)}$ and we denote its image by $\G'_{\s_j(t)}$. $\G'_{\s_j(t)}$ stabilizes $\H'$. Then $\r_j(t)$ can be viewed as defined on the same domain $\H'$, i.e., $\r_j(t):\H'\to M$, but invariant under different Fuchsian groups action, i.e.,
\[\forall\a\in\G'_{\s_j(t)},\quad\r_j(t)\circ\a=\r_j(t).\]

Now we introduce the notion of \textit{collections of disjoint balls} on $\S_{\s(t)}$. Here we use $\B=\cup_{i=1}^nB_i$ to denote a finite collection of disjoint geodesic balls on $\S_{\s(t)}$, those geodesics balls can either lie in the interior of $\S_{\s(t)}$ or intersect with $\partial\S_{\s(t)}$ orthogonally, with the radii of each ball less than the injective radius of the center of that ball on $\S^d_{\s(t)}$. If $B$ is an interior ball then the center of $B$ lies in $\S_{\s(t)}$. If $B$ is a ball that intersects the boundary $\partial\S_{\s(t)}$ orthogonally, then the center of it lies on $\partial\S_{\s(t)}$, and we take the injective radius of its center as on $\S_{\s(t)}^d$. Taking a ball $B\in\B$ with radius $r_B$ , we will use a sub-geodesic ball with the same center but with radius $\rho r_B$, which we denote by $\rho B$.

For an interior geodesic ball $B$ with hyperbolic metric of curvature $-1$ can always be pulled back to the Poincar\'{e} disk $(D,ds^2_{-1}=\frac{|dx^2|}{(1-|x|^2)^2})$, such that the center of $B$ goes to the center of $D$. Then $B$ can be viewed as a disk $B(0,r_B^0)$ \footnote{$B(0,r^0)$ is denoted to be a disk center at $0$ with Euclidean radius $r_0$} in $D$ with hyperbolic metric $ds^2_{-1}$, where $r_B^0$ is the Euclidean radius of the image of $B$ and the radius w.r.t. hyperbolic metric is $r_B=\int_0^{r_B^0}\frac{1}{1-t^2}dt=\tanh^{-1}(r_B^0).$
The hyperbolic metric is conformal and uniformly equivalent to the Euclidean metric $ds_0^2=|dx|^2$ on $B$. Here \textit{uniformly equivalent} means $ds_0^2\leq ds^2_{-1}\leq Cds_0^2$ for some constant $C>1$. There exists a small number:
\begin{equation}\label{rho}
r_0:=\tanh^{-1}(\frac{1}{2}),
\end{equation}
such that if we restrict the radius $r_B$ of $B$ with $r_B\leq r_0$, we can choose the constant $C=\frac{16}{9}$. Then if we consider $\frac{1}{4}B$, under the Euclidean metric $ds_0^2$, the radius of $\frac{1}{4}B$ is less than $\frac{1}{2}r_B^0$, i.e., $\frac{1}{4}B\subset B(0,\frac{1}{2}r_B^0)$. Later on, we will always assume that the geodesic balls have their radii bounded from above by $r_0$.

Similarly, for a geodesic ball $B$ of $\S_{\s_j(t)}$ which intersects the boundary $\partial\S_{\s_j(t)}$ orthogonally. Since $\pi^{-1}_{\s_j(t)}(\partial\S_{\s_j(t)})$ is fixed for all $t\in[0,1]$, and all $j\in\N$, and by Claim \ref{border geodesic} we know that $\partial\S_{\s_j(t)}$ is a geodesic in $S^d_{\s_j(t)}$. So we can assume one of the component of $\partial\S_{\s_j(t)}$ can be pulled back to the imaginary axis on $\H$ by the quotient map $\pi_{\s_j(t)}.$ Thus we can again pull back the geodesic ball to the Poincar\'{e} disk, such that the center of $B$ goes to the center of $D$ and the boundary component $\partial\S_{\s_j(t)}$ on which $B$ intersects orthogonally with can be pulled back to the real axis on the Poincar\'{e} disk. So that $B$ corresponds to the intersection of the geodesic ball centered at the center of Poincar\'{e} disk with upper half plane, i.e., $D^{+}=D\cap\{(x_1,x_2)|x_2>0\}$. We also assume that the geodesic balls have their radii bounded above by $r_0.$

\subsection{Continuity of the harmonic replacement map}
\begin{theorem}\cite[Theorem 3.1]{CD}\label{energy convexity}(Energy convexity for weakly harmonic map)
There exists a constant $\e_1>0$ (depending only on the manifold $M$) such that for all maps $u,v\in W^{1 ,2}(B,M)$, if $E(u)\leq\e_1$ and $v$ is weakly harmonic with the same boundary value as $u$, then we have:
\begin{equation}\label{convexity1}
\int_B|\n u|^2-\int_B|\n v|^2\geq\frac{1}{2}\int_B|\n u-\n v|^2.
\end{equation}
\end{theorem}
Recall that for a geodesic ball $B$ which interests the boundary $\partial\S_{\s_j(t)}$ orthogonally, it can be pulled back to the upper half disk on the Poincar\'{e} disk. We fix some notations here.
\begin{itemize}
    \item $D_s$ denotes the disk centered at origin with radius $s$.
    \item $D^+_s$ denotes the upper half disk with radius $s$ centered at the origin.
    \item $\l^A_s=\{(r,\theta):r\equiv s,\:0\leq\theta\leq\pi\}.$
    \item $\l^C_s=\{(r,\theta),0\leq r\leq s,\:\theta=0\text{ or }\pi\}.$
\end{itemize}
\begin{theorem}\cite[Theorem 2.5]{LAX}\label{regularity}
There exists a constant $\e_1>0$ depending only on $M$ and $N$ such that if $v\in W^{1,2}(D^+,M)$ is a weakly harmonic map with its energy bounded by $\e_1$, and with mixed Dirichlet boundary on $\l^A$ and free boundary $v|_{\l^C}\subset N$, then for any $x\in D^+\cup(\l^C)^0$ we have
\begin{equation}\label{epsilonregularity}
    |\nabla v|(x)\leq C\frac{\sqrt{\e_1}}{1-|x|},
\end{equation}
for some constant $C>0$ that only depends on $M$ and $N$.
\end{theorem} 
\begin{theorem}\cite{LAX}\cite{PR}(Energy convexity for weakly harmonic maps with mixed Dirichlet and free boundaries)\label{energy convexity free boundary} For $u,v\in W^{1,2}(D^+,M)$ with $u|_{\l^A}=v|_{\l^A}$ and 
$u|_{\l^C}\subset N$, $v|_{\l^C}\subset N$, energy of $u$ is bounded by $\e_1>0$ given in Theorem \ref{regularity}, and $v$ is a weakly harmonic map with partial free boundary. Then we have the energy convexity:
\begin{equation}\label{convexity2}
\int_{D^+}|\n u|^2-\int_{D^+}|\n v|^2\geq\frac{1}{2}\int_{D^+}|\n u-\n v|^2.
\end{equation}
\end{theorem}
\begin{remark}
Although Theorem \ref{energy convexity} and Theorem \ref{energy convexity free boundary} are both formulated using the standard Euclidean metric $ds_0^2$  and flat connection. We can still have the same result if we take another metric $ds^2$ on $B$ which is conformal to $ds_0^2$. Therefore, if we take the standard hyperbolic metric $ds_{-1}^2$ on a Pincar\'{e} disk. \eqref{convexity1} and \eqref{convexity2} both sill hold by considering the geodesic balls on the Pincar\'{e} disk with radius bounded by \eqref{rho}, and changing the flat connection to the connection of $ds_{-1}^2$
\end{remark}
\begin{corollary}\cite[Corollary 3.4]{CD}\cite[Corollary 4.6]{minmaxgenus}\label{6}
Let $\e_1>0$ be given in Theorem  \ref{energy convexity}. Suppose $u\in C^0(\bar{B})\cap W^{1,2}(B)$ with energy $E(u)\leq\e_1$, then there exists a unique energy minimizing map $v\in C^0(\bar{B})\cap W^{1,2}(B)$ with the same boundary value as $u$. Set 
\[\M=\{u\in C^0(\bar{B})\cap W^{1,2}(B):E(u)
\leq\e_1\}.\]
If we denote the map $v$ by $H(u)$, then the map $H:\M\to\M$ is continuous w.r.t. the norm on $C^0(\bar{B})\cap W^{1,2}(B).$

Suppose that $\{u_j\}_{j\in\N}$ are defined on a ball $B_{1+\e}$
with $E(u_j)\leq\e_1$, $\forall j\in\N$, and 
\[\lim_{j\to\infty}u_j=u,\quad\text{in }C^0(\bar{B}_{1+\e})\cap W^{1,2}(B_{1+\e}).\]
Then for any sequence $r_j\to 1$, let $\{w_j\}_{j\in\N}$ and $w$ be the energy minimizing mappings which coincide with $\{u_j\}_{j\in\N}$ and $u$ outside of $r_jB$ and $B$, then we have
\[w_j\to w,\quad\text{in }C^0(\bar{B}_{1+\e})\cap W^{1,2}(B_{1+\e}).\]
\end{corollary}

\begin{corollary}\cite[Proposition 4.1]{PR}\cite[Theorem 6.1]{LAX}\label{replacement conti}\label{7}
Let $\e_1>0$ be given in Theorem \ref{energy convexity free boundary}. Suppose $u\in C^0(\bar{D}^+)\cap W^{1,2}(D^+)$ with energy $E(u)\leq\e_1$, then there exists a unique energy minimizing map $v\in C^0(\bar{D}^+,M)\cap W^{1,2}(D^+,M)$ with $v|_{\l^A}=u|_{\l^A}$ and $v|_{\l^C}\subset N$. Set 
\[\M=\{u\in C^0(\bar{D}^+,M)\cap W^{1,2}(D^+,M):E(u)
\leq\e_1,\text{ and }u|_{\l^C}\subset N\}.\]
If we denote the map $v$ by $H(u)$, then the map $H:\M\to\M$ is continuous w.r.t. the norm on $C^0(\bar{D}^+,M)\cap W^{1,2}(D^+,M).$ Moreover, there exists a constant $C_1(M,N,\e_1)>0$ such that for $u_1,u_2\in\M$ we have
\begin{equation}\label{continuity of harmonic replacement}
    |E(H(u_1))-E(H(u_2))|\leq C_1(\|u_1-u_2\|_{C^0}+\|\nabla(u_1-u_2)\|_{L^2}).
\end{equation}
\end{corollary}
\begin{definition}
Let $u\in\M$, we call $v\in\M$ a free boundary harmonic replacement of $u$ if $v$ is harmonic and $u|_{\l^A}=v|_{\l^A}$. The existence of $v$ follows from \cite[Theorem 3.4]{LAX} and the uniqueness follows form Theorem \ref{energy convexity free boundary}.
\end{definition}
We also need the following extension of Corollary \ref{replacement conti}. The proof is a simple adaptation of \cite[Corollary 3.4]{CD}, \cite[Corollary 4.2]{minimaltorus}, and \cite[Proposition 4.1]{PR}.
\begin{corollary}\label{8}
Suppose that $\{u_j\}_{j\in\N}$ are defined on $D^+_{1+\e}$
with energy less than the constant $\e_1>0$ given in Theorem \ref{energy convexity free boundary}, $\forall j\in\N$, and 
\[\lim_{j\to\infty}u_j=u,\quad\text{in }C^0(\bar{D}^+_{1+\e})\cap W^{1,2}(D^+_{1+\e}).\]
Then for any sequence $r_j\to 1$, let $\{w_j\}_{j\in\N}$ and $w$ be the energy minimizing mappings such that
\begin{itemize}
    \item $w_j$ coincide with $u_j$ outside of $D^+_{r_j}$, $\forall j\in\N$,
    \item $w$ coincides with $u$ outside of $D^+$,
    \item $w_j|_{\l_{r_j}^C}\subset N$,
    \item $w|_{\l^C}\subset N$,
\end{itemize}
then we have
\[w_j\to w,\quad\text{in }C^0(\bar{D}^+_{1+\e})\cap W^{1,2}(D^+_{1+\e}).\]
\end{corollary}
\begin{proof}
\begin{claim}\label{5.8}
Let $\{\tilde{w}_j\}$ be the sequence of energy minimizing map such that 
\begin{itemize}
    \item $\tilde{w}_j$ coincide with $u$ outside of $D^+_{r_j}$, $\forall j\in\N$,
    \item $\tilde{w}_j|_{\l_{r_j}^C}\subset N$,
\end{itemize}
then we have 
\[\tilde{w}_j\to w,\quad\text{in }C^0(\bar{D}^+_{1+\e})\cap W^{1,2}(D^+_{1+\e}).\]
\end{claim}
\begin{proof}
Since the energy of $u$ is bounded by $\e_1>0$ given in Theorem \ref{energy convexity free boundary} and Theorem \ref{regularity}, we know by Theorem \ref{regularity} that $\tilde{w}_j$ have uniform inner $C^{2,\a}$ bounds on $D^+\cup(\l^C)^0$. So $\forall r<1$, $\tilde{w}_j\to w'$ in $C^{2,\a}(D^+_r)$ and $w'$ is a harmonic map on $D^+$ with $w'|_{\l^C}\subset N.$ By scaling argument, we can show that there is no energy concentration points near $\l D^+$. So $\tilde{w}_j\to w'$ in $W^{1,2}(B_{1+\e})$. We also know from \cite[Proposition 4.1]{PR} that $\tilde{w}_j$ is equicontinuous on $\bar{D}_{r_j}^+$, hence $\tilde{w}_j\to w$ in $C^0(\bar{D}^+_{1+\e})$. Theorem \ref{energy convexity free boundary} implies the uniqueness of small energy harmonic map with mixed Dirichlet and free boundaries, thus we have $w'=w.$
\end{proof}
Recall that we assume the Riemannian manifold $(M,g)$ is isometrically embedded in $\R^N$. In the following, we let $\delta>0$ be such that for the open neighborhoods
\[N_{\delta}=\big\{x\in\R^N;d(x,N)<\delta\big\}\text{ and }M_{\delta}=\big\{x\in\R^N;d(x,N)<\delta\big\}\]
of the submanifolds $N$ and $M$ in $\R^N$ respectively, there are smooth nearest point projection maps $\Pi_N:N_\delta\to N$ on $N$ and $\Pi_M:M_\delta\to M$ on $M$ such that
\begin{equation}\label{nppbdd}
\sup_{N_\delta}\|D\Pi_N\|\leq 2,
\end{equation}
and
\begin{equation}\label{nppbdd2}
\sup_{M_\delta}\|D\Pi_M\|\leq 2,\forall x\in M_{\delta},\|D\Pi_M\|\leq 1+C|x-\Pi_M(x)|.
\end{equation}
Moreover, we can extend $\Pi_N:N_{\frac{\delta}{2}}\to N$ to $\tilde{\Pi}_N:\R^N\to\R^N$ by a smooth cut-off function $\chi$ such that $\chi=1$ on $N_{\frac{\delta}{2}}$ and $\chi=1$ on $\R^{N}\setminus N_{\delta}$. We can set $\tilde{\Pi}_N=\chi\Pi_N$.
 
Now we consider
\[\bar{w}_j:=\tilde{w}_j-\tilde{\Pi}_N(\tilde{w}_j)+\tilde{\Pi}_N(\tilde{w}_j+u_j-u)+u_j-\tilde{\Pi}_N(u_j)-(u-\tilde{\Pi}_N(u))\]
in order to get some comparison map close to $\tilde{w}_j$ as $u_j$ is close to $u$. We can see that $\tilde{w}_j$ coincide with $u_j$ outside of $D^+_{r_j}$.
\begin{equation}\label{conincide}
    \forall x\in\l^C_{r_j}, \bar{w}_j(x)\in N\text{ and }\forall x\in\l^A_{r_j},\bar{w}_j(x)=u_j(x).
\end{equation}
Since $u_j\to u$ in $C^0(\bar{D}^+_{1+\e})\cap W^{1,2}(D^+_{1+\e})$, we can assume that $\|u_j-u\|_{C^0}\leq\frac{\delta}{2}$, so that $\forall x\in\l^C_{r_j}$, we have
$\tilde{w}_j+u_j-u\in N_{\frac{\delta}{2}}$, and thus
\[\bar{w}_j=\tilde{\Pi}_N(\tilde{w}_j+u_j-u)=\Pi_N(\tilde{w}_j+u_j-u).\]
Now, we can set 
\begin{equation}\label{comparison}
    \hat{w}_j=\Pi_M(\bar{w}_j).
\end{equation}
$\hat{w}_j$ is well-defined. Since
\begin{equation}\label{15}
\begin{split}
d(\bar{w}_j,M)\leq&|\bar{w}_j-\tilde{w}_j|\\
    \leq&|-\tilde{\Pi}_N(\tilde{w}_j)+\tilde{\Pi}_N(\tilde{w}_j+u_j-u)+u_j-\tilde{\Pi}_N(u_j)-(u-\tilde{\Pi}_N(u))|\\
    \leq&|\tilde{\Pi}_N(\tilde{w}_j+u_j-u)-\tilde{\Pi}_N(\tilde{w}_j)|\\
    &+|(id-\tilde{\Pi}_N)(u_j)-(id-\tilde{\Pi}_N)(u)|\\    
\end{split}
\end{equation}
so that 
\begin{equation}\label{16}
d(\bar{w}_j,M)\leq(\sup\|D\tilde{\Pi}_N\|+\sup\|D(id-\tilde{\Pi}_N)\|)\|u_j-u\|_{C^0},    
\end{equation}
since $u_j\to u$ in $C^0(\bar{D}^+_{1+\e})\cap W^{1,2}(D^+_{1+\e})$,and by \eqref{nppbdd} we can assume that $\|u_j-u\|_{C^0}$ is small enough so that $d(\bar{w}_j,M)\leq\delta$ thus $\hat{w}_j$ is well-defined.
\begin{claim}
\[w_j\to w,\quad\text{in } W^{1,2}(D^+_{1+\e}).\]
\end{claim}
\begin{proof}
First we show that
$\hat{w}_j-\tilde{w}_j\to 0$ in $W^{1,2}(D^+_{1+\e})$. 
We start by estimating the $L^2$-norm of $\nabla\bar{w}_j-\nabla\tilde{w}_j$ on $D^+_{1+\e}$, we have
\begin{equation}
\begin{split}
    \nabla\bar{w}_j-\nabla\tilde{w}_j=&(D\tilde{\Pi}_N(\tilde{w}_j+u_j-u)-D\tilde{\Pi}_N(\tilde{w}_j))\cdot\nabla\tilde{w}_j\\
    &+D\tilde{\Pi}_N(\tilde{w}_j+u_j-u)\cdot\nabla(u_j-u)\\
    &+((id-D\tilde{\Pi}_N)(u_j)-(id-D\tilde{\Pi}_N)(u))\cdot\nabla u_j\\
    &+(id-D\tilde{\Pi}_N)(u)\cdot\nabla(u_j-u),
\end{split}    
\end{equation}
so that
\begin{equation}\label{18}
    \begin{split}
        |\nabla\bar{w}_j-\nabla\tilde{w}_j|\leq&\sup\|D^2\tilde{\Pi}_N\|(|\nabla\tilde{w}_j|+|\nabla u_j|)\|u_j-u\|_{C^0(D^+_{1+\e})}\\
        &+(\sup\|D\tilde{\Pi}_N\|+\sup\|id-D\tilde{\Pi}_N\|)|\nabla(u_j-u)|.
    \end{split}
\end{equation}
Since $\hat{w}_j=\Pi_{M}(\bar{w}_j)$, by \eqref{nppbdd2} we have
\begin{equation}\label{19}
|\nabla\hat{w}_j-\nabla\bar{w}_j|\leq C|(\bar{w}_j-\hat{w}_j)\nabla\bar{w}_j|. 
\end{equation}
\eqref{15} implies that 
\begin{equation}\label{20}
    \begin{split}
     |\bar{w}_j-\tilde{w}_j|&\leq|\tilde{\Pi}_N(\tilde{w}_j+u_j-u)-\tilde{\Pi}_N(\tilde{w}_j)|+|(id-\tilde{\Pi}_N)(u_j)-(id-\tilde{\Pi}_N)(u)|\\
     &\leq\sup\|D\tilde{\Pi}_N\|\|u_j-u\|_{C^0(D^+_{1+\e})}.
    \end{split}
\end{equation}
By \eqref{16} we have that
\begin{equation}\label{21}
\int_{D^+_{1+\e}}|\bar{w}_j-\hat{w}_j|^2\leq(\sup\|D\Pi_N\|+\sup\|D(id-\tilde{\Pi}_N)\|)^2\int_{D^+_{1+\e}}|u_j-u|^2.    
\end{equation}
Combining \eqref{18}\eqref{19}\eqref{20}\eqref{21}, we have that \[\hat{w}_j-\tilde{w}_j\to 0\text{ in }W^{1,2}(D^+_{1+\e}),\]
as $u_j-u\to 0$ in $C^0(\bar{D}^+_{1+\e})\cap W^{1,2}(D^+_{1+\e})$. Claim \ref{5.8} implies that
\[\tilde{w}_j\to w,\quad\text{in }C^0(\bar{D}^+_{1+\e})\cap W^{1,2}(D^+_{1+\e}).\]
So we have $\hat{w}_j\to w$ in $W^{1,2}(D^+_{1+\e})$. Recall that $w_j$ is the energy minimizing map that coincides with $u_j$ outside of $D^+_{r_j}$ and $\tilde{w}_j$ is the energy minimizing map that coincides with $u$ outside of $D^+_{r_j}$. Thus by \eqref{continuity of harmonic replacement} we have 
$|E(\tilde{w}_j)-E(w_j)|\to 0$ on $D^+_{1+\e}$, hence $|E(\hat{w}_j)-E(w_j)|\to 0$ on $D^+_{1+\e}$. Since $\hat{w}_j$ coincides with $w_j$ outside of $D^+_{r_j}$, Theorem \ref{energy convexity free boundary} implies that $w_j-\hat{w}_j\to 0$ in $W^{1,2}(D^+_{r_j})$. So
\[\int_{D^+_{1+\e}}|\nabla w_j-\nabla w|^2=\int_{D^+_{r_j}}|\nabla w_j-\nabla w|^2+\int_{D^+_{1+\e}\setminus D^+_{r_j} }|\nabla u_j-\nabla w|^2\to 0.\]
Hence we have 
\[w_j\to w,\quad\text{in } W^{1,2}(D^+_{1+\e}).\]
\end{proof}
To show the $C^0(\bar{D}^+_{1+\e})$ convergence, we know that $w_j$ are equicontinuous near $\l^A$ by the equicontinuity of $u_j$. By Theorem \ref{regularity} we know that $w_j$ have uniform inner $C^{2,\a}$ bounds on $D^+\cup(\l^C)^0$. So we have $w_j\to w$ in $C^0(\bar{D}^+_{1+\e})$ up to subsequence.
\end{proof}
\begin{remark}
If we use the interior geodesic ball or the geodesic ball which intersects $\l\S_0$ orthogonally of radius $r\leq r_0$ on a hyperbolic surface $\S_0$ with Poincar\'{e} metric, all the results of Corollary \ref{6}, Corollary \ref{7}, and Corollary \ref{8} hold. This is because that the Poincar\'{e} metric $ds_{-1}^2$ is conformal and uniformly equivalent to the flat metric $ds_0^2$, so harmonic maps w.r.t. $ds_0^2$ are also harmonic maps w.r.t. $ds_{-1}^2$, and the $C^0$ and $W^{1,2}$ norms of a fixed map w.r.t. $ds_{-1}^2$ are uniformly equivalent to those w.r.t. $ds_0^2$. 
\end{remark}
In the above section we've used the symbol $D^+_r$ for the upper half disk with radius $r$ w.r.t. flat metric $ds_0^2$. Since for a given bordered Riemann surface $\S$ with genus at least one (so that its double $\S^d$ carries a hyperbolic structure), we can pull back the geodesic ball $B$ which intersects the boundary $\l\S$ orthogonally to the Poincar\'{e} disk such that the center of $B$ goes to the center of $D$ and the boundary component $\l\S$ on which $B$ intersects orthogonally with can be pulled back to the real axis on the Poincar\'{e} disk. So that $B$ corresponds to the intersection of the geodesic ball centered at the center of Poincar\'{e} disk with upper half plane. In the following content, in abuse of notation, we will also denote the symbol $D^+_r$ as the geodesic ball which intersects the boundary $\partial\S$ orthogonally, where $r<r_0$ is the radius w.r.t. $ds^2_{-1}$.

\subsection{Comparison results of successive harmonic replacement}
Let $\e_1>0$ be as in Theorem \ref{regularity}. We adopt the following notation: given $u\in C^0(\bar{\S_0},M)\cap W^{1,2}(\S_0,M)$, and a finite collection $\B$ of disjoint geodesic balls on $\S_0$ so that the radius of each ball $B\in\B$ is less than the injective radius of the center of $B$ on $\S^d_0$, and also bounded by $r_0$ given by \eqref{rho}, the energy of $u$ on $\cup_{\B}B$ is at most $\frac{\e_1}{3}$, let $H(u,\B)$ denote the map that coincides with $u$ on $\S_0\setminus\cup_{\B}B$ and equal to the free boundary harmonic replacements of $u$ on $\cup_{\B}B$. Note we also have $H(u,\B)(\l\S_0)\subset N$, and we will call $H(u,\B)$ the free boundary harmonic replacement of $u$ on $\B$. Given two such collections $\B_1$ $\B_2$, we use $H(u,\B_1,\B_2)$ to denote $H(H(u,\B_1),\B_2)$. Recall that for $\p\in(0,1]$, $\p\B$ will denote the collection of concentric balls with radii that are shrunk by the factor $\p$. We have the following energy comparison result for  
\begin{lemma}\label{5.13}
Fix a bordered Riemann surface $\S_0$. Given $u\in C^0(\bar{\S_0},M)\cap W^{1,2}(\S_0,M)$. Let $\B_1$ and $\B_2$ be two finite collection of geodesic balls on $\S_0$, with the radius of each ball $B$ less than the injective radius of the center of $B$ on the double $\S_0^d$ and $r_0$ given by \eqref{rho}. If $E(u,\B_i)\leq\frac{\e_1}{3}$ for $i=1,2,$ with $\e_1>0$ given in Theorem \ref{energy convexity free boundary}, then there exists a constant $k>0$ depending on $M$ and $N$ such that 
\begin{equation}\label{24}
    E(u)-E(H(u,\B_1,\B_2))\geq k\Big(E(u)-E(H(u,\frac{1}{4}\B_2))\Big)^2,
\end{equation}
and for $\p\in[\frac{1}{64},\frac{1}{4}]$,
\begin{equation}\label{25}
    \begin{split}
        \frac{1}{k}(E(u)-E(H(u,\B_1)))^{\frac{1}{2}}+E(u)-&E(H(u,4\p\B_2))\\
        &\geq E(H(u,\B_1))-E(H(u,\B_1,\p\B_2)).
    \end{split}
\end{equation}
\end{lemma}
\begin{remark}
Since \eqref{24} and \eqref{25} are all conformal invariant, the proof in the Euclidean metrics implies that in hyperbolic metrics. We will use the Euclidean metric in the proof of Lemma \ref{5.13} which is conformal to the hyperbolic metric on each of the geodesic balls.
\end{remark}
We need the following Lemma \ref{comparison1} and Lemma \ref{comparison2} to construct the  comparison maps. Lemma \ref{comparison1} addresses the case of interior geodesic ball $B_R$ in $\S_0$ and Lemma \ref{comparison2} addresses the case of the geodesic ball $D^+_R$ which intersects the boundary $\l\S_0$ orthogonally.
\begin{lemma}\cite[Lemma 3.14]{CD}\label{comparison1}
There exists $\delta>0$ and a large constant $C$ depending on $M$ such that for any $f,g\in C^0\cap W^{1,2}(\l B_R,M)$, where $B_R$ is an interior geodesic ball, if $f$, $g$ are equal at some point on $\l B_R$ and
\[R\int_{\l B_R}|f'-g'|^2\leq\delta^2,\]
then we can find some $\p\in(0,\frac{1}{2}R]$ and a map $w\in C^0\cap W^{1,2}(B_R\setminus B_{R-\p},M)$ with $w|_{B_R}=f$, $w|_{B_{R-\p}}=g$, which satisfies the following estimates
\begin{equation}
\int_{B_R\setminus B_{R-\p}}|\nabla w|^2\leq C(R\int_{\l B_R}|f'|^2+|g'|^2)^{\frac{1}{2}}(R\int_{\l B_R}|f'-g'|^2)^{\frac{1}{2}}.    
\end{equation}
\end{lemma} 
\begin{lemma}\cite[Lemma 6.5]{LAX}\cite[Lemma 4.2]{PR}\label{comparison2}
There exists $\delta>0$ and a large constant $C$ depending on $M$ and $N$ such that for any $f,g\in C^0\cap W^{1,2}(\l^A_R,M)$ with $\{f(0),f(\pi),g(0),g(\pi)\}\subset N$. If $f$ and $g$ agree at one point on $\l^A_R$ and satisfy:
\[R\int_{ \l_R^A}|f'-g'|^2\leq\delta^2,\]
then we can find some $\p\in(0,\frac{R}{2}]$ and a map $w\in C^0\cap W^{1,2}(D^+_R\setminus D^+_{R-\p},M)$ with $w|_{\l^A_{R-\p}}=f$, $w|_{\l^A_{R}}=g$, and $w|_{\l^C_R\setminus \l^C_{R-\p}}\subset N$ such that the following equation holds
\begin{equation}
\int_{D^+_R\setminus D^+_{R-\p}}|\nabla w|^2\leq C(R\int_{\l^A _R}|f'|^2+|g'|^2)^{\frac{1}{2}}(R\int_{\l ^A_R}|f'-g'|^2)^{\frac{1}{2}}.    
\end{equation}
\end{lemma}

\begin{proof}(of Lemma \ref{5.13})
The case of which $\B_1$ and $\B_2$ contain only interior geodesic balls was addressed in \cite[Lemma 4.8]{minmaxgenus}. So we focus on the case of $\B_1$ and $\B_2$ are both finite collection of geodesic balls which intersect the boundary $\l\S_0$ orthogonally.

Let $\B_1=\{B_{\a}^1\}$ and $\B_2=\{B_{\b}^2\}$. We will divide $\B_2$ into two disjoint subsets $\B_{2,+}$ and $\B_{2,-}$, set
\[\B_{2,+}:=\Big\{B_{\b}^2\in\B_2:\frac{1}{4}B_{\b}^2\subset B_{\a}^1\text{ for some } B_{\a}^1\in\B_1,\text{ or }\frac{1}{4}B_{\b}^2\cap\B_1=\emptyset\Big\},\]
and $\B_{2,-}:=\B_2\setminus\B_{2,+}$. For balls $\frac{1}{4}B_{\b}^2\cap\B_1=\emptyset$, we have
\[\sum_{\{\frac{1}{4}B_{\b}^2\cap\B_1=\emptyset\}}E(u)-E(H(u,\frac{1}{4}B_{\b}^2))\leq E(u)-E(H(u,\B_1,\cup_{\frac{1}{4}B_{\b}^2\cap\B_1=\emptyset}B_{\b}^2)).\]
For $\frac{1}{4}B_{\b}^2\subset B_{\a}^1$, we have
\begin{equation}
    \begin{split}
        \int_{\cup_{\frac{1}{4}B_{\b}^2\subset B_{\a}^1}B^2_{\b}}|\nabla u|^2-&|\nabla H(u,\frac{1}{4}B_{\b}^2)|^2\leq\int_{\cup_{\frac{1}{4}B_{\b}^2\subset B_{\a}^1}B^2_{\b}}|\nabla u|^2-|\nabla H(u,\B_1,\B_{\b}^2)|^2\\
        \leq&\int_{\cup_{\frac{1}{4}B_{\b}^2\subset B_{\a}^1}B^2_{\b}}|\nabla u|^2-|\nabla H(u,\B_1)|^2\\
        &+\int_{\cup_{\frac{1}{4}B_{\b}^2\subset B_{\a}^1}B^2_{\b}}|\nabla H(u,\B_1)|^2-|\nabla H(u,\B_1,B^2_{\b})|^2.
    \end{split}
\end{equation}
By Theorem \ref{energy convexity free boundary} we know that
\begin{equation}
    \begin{split}
\int_{\cup_{\frac{1}{4}B_{\b}^2\subset B_{\a}^1}B^2_{\b}}|\nabla u|^2-|\nabla H(u,\B_1)|^2&\leq\int_{\cup_{\frac{1}{4}B_{\b}^2\subset B_{\a}^1}B^2_{\b}}|\nabla u-\nabla H(u,\B_1)|^2\\
&\leq 4(E(u)-E(H(u,\B_1))).
    \end{split}
\end{equation}
So combining the above estimates together we the following inequality:
\begin{equation}\label{30'}
E(u)-E(H(u,\frac{1}{4}\B_{2,+}))\leq C(E(u)-E(u,\B_1,\B_{2,+})).
\end{equation}
Now we consider $\B_{2,-}$. For $B^2_{\b}\in\B_{2,-}$ we know that $\frac{1}{4}B^2_{\b}\cap B^1_{\a}\neq\emptyset$ for some $B^1_{\a}\in\B_1$, but We can pull back $B^2_{\b}$ to the upper Poincar\'{e} disk with its center at the origin and denote it by $D^+_{r_B^0}$, where $r_B^0$ is the radius w.r.t. $ds^2_0$. Now let us construct an auxiliary comparison map. Using co-area formula, there exists some $r\in[\frac{3}{4}r_B^0,r_B^0]$ with
\[\int_{\l^A_r}|\nabla u-\nabla u_1|^2\leq\frac{9}{r_B^0}\int^{r^0_B}_{3r^0_B/4}(\int_{\l^A_s}|\nabla u-\nabla u_1|^2)ds\leq\frac{9}{r}\int_{D^+_{r_B^0}}|\nabla u-\nabla u_1|^2,\]
\[\int_{\l^A_r}(|\nabla u_1|^2+|\nabla u|^2)\leq\frac{9}{r_B^0}\int^{r^0_B}_{3r^0_B/4}(\int_{\l^A_s}|\nabla u_1|^2+|\nabla u|^2)ds\leq\frac{9}{r}\int_{D^+_{r_B^0}}|\nabla u|^2+|\nabla u_1|^2.\]
where $u_1=H(u,\B_1)$. We apply Lemma \ref{comparison2} to get $\p\in(0,r/2]$ and $w:D^+_{r}\setminus D^+_{r-\p}\to M$ with $w(r,\theta)=H(u,\B_1)(r,\theta)$ and $w(r-\p,\theta)=u(r,\theta)$ such that
\begin{equation}\label{30}
\begin{split}
\int_{D^+_r\setminus D^+_{r-\p}}|\nabla w|^2\leq& C(r\int_{\l^A _r}|\nabla u|^2+|\nabla u_1|^2)^{\frac{1}{2}}(r\int_{\l ^A_{r
}}|\nabla(u-u_1)|^2)^{\frac{1}{2}}\\   
\leq& C(\int_{D^+_{r_B^0}}|\nabla u|^2+|\nabla u_1|^2)^{\frac{1}{2}}(\int_{D^+_{r_B^0}}|\nabla u-\nabla u_1|^2)^{\frac{1}{2}}.
\end{split}
\end{equation}
The map $x\mapsto H(u,D_r^+)(rx/(r-\p))$ maps $D^+_{r-\p}$ to $M$ and agrees with $w$ on $\l^A_{r-\p}$ (the rescaling here is done w.r.t. the flat metric). So we get the following map
\[
v = \begin{cases} H(u,\B_1), & \text{on } D^+_{r_B^0}\setminus D_r^+ \\
w, & \text{on }D_r^+\setminus D_{r-\p}^+  \\
H(u,D^+_r)(\frac{rx}{r-\p})&\text{on }D^+_{r-\p}
\end{cases}
\]
This new map $v$ gives an upper bound for the energy of $H(u,\B_1,D^+_{r_B^0})$:
\[
\begin{split}
    \int_{D^+_{r_B^0}}|\nabla H(u,\B_1,D^+_{r_B^0})|^2\leq&\int_{D^+_{r_B^0}}|\nabla v|^2\\
    =&\int_{D^+_{r_B^0}\setminus D^+_r}|\nabla H(u,\B_1)|^2+\int_{D^+_r\setminus D^+_{r-\p}}|\nabla w|^2\\
    &+\int_{D^+_r}|\nabla H(u,D^+_r)|^2.
\end{split}
\]
Since $\frac{1}{4}D^+_{r_B^0}\subset D^+_{\frac{1}{2}r_B^0}\subset D^+_r$, combining with the above inequality, we have the following
\[
\begin{split}
\int_{\frac{1}{4}D^+_{r_B^0}}|\nabla u|^2-\int_{\frac{1}{4}D^+_{r_B^0}}|\nabla &H(u,\frac{1}{4}D^+_{r_B^0})|^2\leq\int_{D^+_r}|\nabla u|^2-\int_{D^+_r}|\nabla H(u,D^+_r)|^2\\    
\leq&\int_{D^+_r}|\nabla u|^2-\int_{D^+_{r_B^0}}|\nabla H(u,\B_1,D^+_{r_B^0})|^2\\
&+\int_{D^+_{r_B^0}\setminus D^+_r}|\nabla H(u,\B_1)|^2+\int_{D^+_r\setminus D^+_{r-\p}}|\nabla w|^2\\
\leq&\int_{D^+_r}|\nabla u|^2-\int_{D^+_{r_B^0}}|\nabla H(u,\B_1,D^+_{r_B^0})|^2+\int_{D^+_r\setminus D^+_{r-\p}}|\nabla w|^2\\
&+\int_{D^+_{r_B^0}}|\nabla H(u,\B_1)|^2-\int_{D^+_r}|\nabla H(u,\B_1)|^2.\\
\end{split}
\]
Combining with \eqref{30} we have
\[
\begin{split}
\int_{\frac{1}{4}D^+_{r_B^0}}|\nabla u|^2-\int_{\frac{1}{4}D^+_{r_B^0}}|\nabla H(u,\frac{1}{4}D^+_{r_B^0})|^2
\leq&\int_{D^+_r}|\nabla u|^2-\int_{D^+_r}|\nabla H(u,\B_1)|^2\\
+&\int_{D^+_{r_B^0}}|\nabla H(u,\B_1)|^2-\int_{D^+_{r_B^0}}|\nabla H(u,\B_1,D^+_{r_B^0})|^2\\
+&C(\int_{D^+_{r_B^0}}|\nabla u|^2+|\nabla u_1|^2)^{\frac{1}{2}}(\int_{D^+_{r_B^0}}|\nabla u-\nabla u_1|^2)^{\frac{1}{2}}.
\end{split}
\]
With Theorem \ref{energy convexity free boundary}, summing the above inequality over $\B_{2,-}$, and the assumption that all the maps have energy less than $\e_1>0$, we can get the following inequality
\[E(u)-E(H(u,\frac{1}{4}B_{2,-}))\leq C'(E(u)-E(H(u,\B_1,\B_2)))^{\frac{1}{2}}.\]
With \eqref{30'} we get the inequality \eqref{24}.

As for the inequality \eqref{25}, we divide $\B_2$ into two disjoint sub-collection $\B_{2,+}$ and $\B_{2,-}$, with
\[\B_{2,+}:=\Big\{B_{\b}^2\in\B_2:\mu B_{\b}^2\subset B_{\a}^1\text{ for some } B_{\a}^1\in\B_1,\text{ or }\mu B_{\b}^2\cap\B_1=\emptyset\Big\},\]
and $\B_{2,-}=\B_{2}\setminus\B_{2,+}$. For $\mu B_{\b}^2\subset B_{\a}^1$ we have $H(u,\B_1,\mu B_{\b}^2)=H(u,\B_1)$, so we need not consider $\mu B_{\b}^2\subset B_{\a}^1$. For $\mu B_{\b}^2\cap\B_1=\emptyset$ we have 
\[
\begin{split}
E(H(u,\B_1))-E(H(u,\B_1,&\mu B^2_{\b}))\\
=E(u)-E(H(u,\mu B^2_{\b}))&\leq E(u)-E(H(u,4\mu B^2_{\b})).   
\end{split}
\]
Thus
\[E(H(u,\B_1))-E(H(u,\B_1,\mu \B_{2,+}))\leq E(u)-E(H(u,4\mu B_{2,+})).\]
\end{proof}
For collection $\B_{2,-}$, we use similar argument. Here we identify $4\mu B_{\b}^2$ with a half disk centered at the origin of the upper half Poincar\'{e} disk, and denote by $D^+_{r_B^0}$. In the construction of $w$, we change the role of $u$ and $H(u,\B_1)$. Let the comparison map be the following
\[
v = \begin{cases} u, & \text{on } D^+_{r_B^0}\setminus D_r^+ \\
w, & \text{on }D_r^+\setminus D_{r-\p}^+  \\
H(u,\B_1,D^+_r)(\frac{rx}{r-\p})&\text{on }D^+_{r-\p}
\end{cases}
\]
The comparison map $v$ provides an upper bound for the energy of $H(u,D^+_{r_B^0})$ on $D^+_{r_B^0}$. Since $\mu B^2_{\b}=\frac{1}{4}D^+_{r_B^0}\subset D^+_{r_B^0}$, by similar argument we can get
\[E(H(u,\B_1))-E(H(u,\B_1,\mu\B_{2,-}))\leq E(u)-E(H(u,4\mu\B_{2,-}))+C(E(u)-E(H(u,\B_1)))^{\frac{1}{2}}.\]
Combining results on $\B_{2,+}$ and $\B_{2,-}$ we ge the inequality \eqref{25}.
\subsection{Construction of the deformation map}
Given $(\r(t),\s(t))\in\tilde{\O}$. Recall that for each $\s(t)$ there exists a normalized Fuchsian group $\G_{\s(t)}$ of its double $\S^d_{\s(t)}$, and we denote the quotient map by 
\[\pi_{\s(t)}:\H\to\H/\G_{\s(t)}=\S^d_{\s(t)}.\]
We denote the injective radius of $\S^d_{\s(t)}$ by $r_{\s(t)}$. Fix a time parameter $t\in(0,1)$. Suppose that $B$ is a geodesic ball on $\S_{\s(t)}$, can be either an interior ball or the geodesic ball that intersects $\l\S_{\s(t)}$ orthogonally, with its radius $r_B$ less than the injective radius of the center of $B$ on $\S^d_{\s(t)}$. As discussed in Claim \ref{invariant of boundary}, we can fix a simply connected component $\H'$ of $\H/\pi^{-1}_{\s(t)}(\l\S_{\s(t)})$, which is invariant of all $t\in[0,1]$, and view $\r(t)$ as being lifed up to $\H'$ by $\pi^{-1}_{\s(t)}$ restricted on $\S_{\s(t)}$. We pick one connected component of the pre-image $\pi^{-1}_{\s(t)}(B)$, and denote the connected componenet still by $B$. 
\begin{description}
\item[1]{$B$ is an interior geodesic ball of $\S_{\s(t)}$} \\
$B$ has radius $r_B$ w.r.t. the hyperbolic metric $ds^2_{-1}$ of $\H'\subset\H$. Moreover, $B$ is a standard ball in $\H'$ w.r.t. the flat metric $ds_0^2$. By the continuity of $\s:[0,1]\to\T(\S_0)$, for parameter $s$ sufficiently close to $t$, the image of $B\in\H'$ under $\pi_{\s(s)}:\H'\to\S_{\s(s)}$ is also a geodisic ball with radius less than the injective radius of the center of that ball on $\S_{\s(s)}$.  
\item[2]{$B$ is a geodesic ball that intersects $\l\S_{\s(t)}$ orthogonally} \\
Since $\l\S_{\s(t)}$ is a geodesic on $\S^d_{\s(t)}$ and by Claim \ref{invariant of boundary} $\pi^{-1}_{\s(t)}(\l\S_{\s(t)})$ is fixed $\forall t\in[0,1]$, we can assume without loss of generality that the imaginary axis is one of the connected components of the pre-image $(i\R)\cap\H\in\pi^{-1}_{\s(t)}(\l\S_{\s(t)}),$ for all $t\in[0,1]$. So we can also assume one of the connected components of the pre-image $\pi^{-1}_{\s(t)}(B)$ is the geodesic ball centered on $(i\R)\cap\H$ and intersects $(i\R)\cap\H$ orthogonally. The continuity of $\s:[0,1]\to\T(\S_0)$ implies the continuity of $\s\mapsto \S^d_{\s(t)}$. By the continuity of $\s\mapsto \S^d_{\s(t)}$ and that the pre-image of the boundary is fixed for all $t\in[0,1]$, we can choose the parameter $s$ sufficiently close to $t$ such that the image of $B$ under $\pi_{\s(s)}:\H'\to\S_{\s(s)}$ is also a geodesic ball with center on the boundary $\l\S_{\s(s)}$, and with radius less than the injective radius of center of that ball on $\S^d_{\s(t)}$.
\end{description}
By the above discussion we know that for each $t\in[0,1]$, there exists $\delta_t>0$, so that $\forall|s-t|<\delta_t$,  $\pi_{\s(s)}(\pi^{-1}_{\s(t)}(B))$ remains a geodesic ball of $\S(\s(s))$ and has its radius bounded by the injective radius of the center of that ball on $\S^d_{\s(s)}$. 

Now we discuss free boundary harmonic replacement of $(\r(t),\s(t))\in\tilde{\O}$ on the geodesic ball $B$. We can choose a smooth cutoff function $\mu:[0,1]\to[0,1]$ such that $\mu(s)=1$ for $|s-t|<\delta_t$ and $\mu(s)=0$ for $|s-t|>\delta$. If we do harmonic replacement of $\r(s):\H'\to M$ on the geodesic ball $\mu(s)B\subset\H'$, Corollary \ref{8} implies that it is also a continuous sweepout in $C^0(
\bar{\S}_{\s(t)},M)\cap W^{1,2}(\S_{\s(t)},M)$, and $\r|_{\l\S_{\s(t)}}\subset N$, i.e. $(\r(t),\s(t))$ with free boundary harmonic replacement on $\mu B\subset\H'$ is still a continuous sweepout in $\tilde{\O}$, which lies in the same homotopy class as $(\r,\s)$.

Given $(\r,\s)\in\tilde{\O}$, we define the maximal improvement for free boundary harmonic replacement on families of disjoint geodesic balls with energy at most $\e$ by
\[e_{\e,\r(t)}=\sup_{\B}\{E(\r(t),\s(t))-E(H(\r(t),\frac{1}{4}\B),\s(t))\},\]
where $\B$ are chosen as any finite collection of disjoint geodesic balls on $\S_{\s(t)}$, including both interior geodesic ball and geodesic ball which intersects the boundary $\S_0$ orthogonally, with the radius of each ball less than the injective radius of the center of the ball on $\S^d_0$ and $\r_0$ as in \eqref{rho}.
\begin{lemma}\label{5.17}
$\forall t\in[0,1]$, if $\r(t)$ isn't harmonic, there exists a neighborhood $I^t\subset(0,1)$ of $t$ depending on $t,\e$ and $\r$ such that $\forall s\in 2I^t$ we have
\[e_{\frac{1}{2}\e,\r(s)}\leq 2e_{\e,\r(t)}.\]
\end{lemma}
\begin{proof}
By the above discussion, the continuity of $\s:[0,1]\to\T(\S_0)$ implies that there exists $\delta_t>0$, such that $\forall|s-t|<\delta_t$,  $\pi_{\s(s)}(\pi^{-1}_{\s(t)}(B))$ remains a geodesic ball of $\S(\s(s))$ and has its radius bounded by the injective radius of the center of that ball on $\S^d_{\s(s)}$. Since $\r(t)$ isn't harmonic, we have $e_{\e,\r(t)}>0$. The continuity of $r(s)$ w.r.t. $s$ implies that there exists a neighborhood $\tilde{I}^t\subset(t-\delta_t,t+\delta_t)$ of $t$ such that for any finite collection of geodesic balls $\B\subset K\subset\H'$, where $K$ is a fixed compact subset in $\H'$, we have the following inequality
\begin{equation}\label{32}
    \frac{1}{2}\int_{\B}|\nabla\r(s)-\nabla\r(t)|^2\leq\min\{\frac{1}{4}e_{\e,\r(t)},\frac{1}{2}\e\},
\end{equation}
where we view $\r(s)$ as being lifted up to $\H'$ by the quotient map $\pi_{\s(s)}:\H\to\S^d_{\s(s)}$ restricted on $\S_{\s(s)}$.

Now we fix $s\in 2\tilde{I}^t$, let $\B_s$ be any finite collection of disjoint goedesic balls $\B_s\subset\S_{\s(s)}$ such that $E(\r(s),\B_s)\leq\frac{1}{2}\e$ and 
\[E(\r(s))-E(H(\r(s),\frac{1}{4}\B_s))\geq\frac{3}{4}e_{\frac{1}{2}\e,\r(s)}.\]
By taking the compact set $K$ sufficiently large we can find a connected componenet of the pre-image in $K$ for each ball in $\B_s$. Then take the image of $\pi_{\s(s)}^{-1}(\B_s)$ under the quotient map $\pi_{\r(t)}:
H\to\S^d_{\s(t)}$, we get another collection of geodesic balls on $S_{\s(t)}$, which we denote by $\B_t$. Since $\B_t$ and $\B_s$ share the same pre-image on $K\subset\H'$, by \eqref{32} we have 
\[E(\r(t),\B_t)\leq E(\r(s),\B_s)+\frac{1}{2}\e\leq\e.\]
Hence $E(\r(t))-E(H(\r(t),\frac{1}{4}\B_t))\leq e_{\e,\r(t)}$.
\begin{equation}
    \begin{split}
        E(\r(s))&-E(H(\r(s),\frac{1}{4}\B_s))\\
    &\leq|E(\r(s))-E(\r(t))|+E(\r(t))-E(H(\r(t),\frac{1}{4}\B_t))\\
    &+|E(H(\r(t),\frac{1}{4}\B_t))-E(H(\r(s),\frac{1}{4}\B_s))|.
    \end{split}
\end{equation}
Corollary \ref{8} implies the continuity of harmonic replacement map when the domain  changes continuously w.r.t. $t$. We can choose a neighborhood $I^t\subset\tilde{I}^t$ such that
\[|E(\r(s))-E(\r(t))|\leq\frac{1}{4}e_{\e,\r(t)},\]
and
\[|E(H(\r(t),\frac{1}{4}\B_t))-E(H(\r(s),\frac{1}{4}\B_s))|\leq\frac{1}{4}e_{\e,\r(t)}.\]
Hence $E(\r(s))-E(H(\r(s),\frac{1}{4}\B_s))\leq 2e_{\e,\r(t)}$. This implies that $e_{\frac{1}{2}\e,\r(s)}\leq 2e_{\e,\r(t)}.$ 
\end{proof}

\begin{lemma}\label{5.18}
There exists a covering $\{I^{t_j}:j=1,...,m\}$ for the parameter space $[0,1]$, and $m$ collections of geodesic balls $\B_{j}\subset\S_{\s(t_j)}$, $j=1,...,m$, where the balls in each collection $B_j$ are pairwise disjoint with the radius of each ball less than the injective radius of the center of that ball on $\S^d_{\s(t_j)}$ and $r_0$ by \eqref{rho}, together with $m$ continuous functions $\mu_j:[0,1]\to[0,1]$, $j=1,...,m$, satisfying:
\begin{itemize}
    \item Each $\mu_j$ is supported in $2I^{t_j}$;
    \item For a fixed $t$, at most two $\mu_j(t)$ are positive, and $E(\r(t),\mu_j(t)\B_j)\leq\frac{1}{3}\e_1$;
    \item If $t\in[0,1]$, such that $E(\r(t),\s(t))\geq W_E(\tilde{\O}_{\r,\s})$, there exists a $j$ such that 
    \[E(\r(t))-E(H(\r(t),\frac{1}{4}\mu_j(t)\B_j))\geq\frac{1}{8}e_{\frac{1}{8}\e_1,\r(t)}\].
\end{itemize}
\end{lemma}
\begin{proof}
Since the energy of $(\r(t),\s(t))$ is continuous w.r.t. $t$, the following set
\[I=\big\{t\in[0,1]|E(\r(t),\s(t))\geq W_E(\tilde{\O}_{\r,\s})/2\big\}\] 
is compact. For each $t$, we choose a finite collection of disjoint geodesic balls $\B_t$ with the radius of each ball less than the injective radius of the center of that ball on $\S^d_{\s(t)}$ and $r_0$ by \eqref{rho} with $\frac{1}{2}\int_{\B_t}|\nabla\r|^2\leq\frac{1}{4}\e_1$, and 
\[E(\r(t))-E(H(\r(t),\frac{1}{4}\B_t))\geq\frac{e_{\e_1/4,\r(t)}}{2}>0.\]
Lemma \ref{5.17} gives an open interval $I^t$ containing $t$ such that for all $s\in 2I^t$, 
\[e_{\frac{\e_1}{8},\r(s)}\leq 2e_{\frac{\e_1}{4},\r(t)},\]
Moreover, we have $\pi_{\s(s)}\circ\pi^{-1}_{\s(t)}(\B_t)=\B_s$ is again a finite collection of disjoint geodesic balls with the radius of each ball less than the injective radius of the center of that ball on $\S^d_{\s(s)}$ and $r_0$ by \eqref{rho}. Using the continuity of $\r$ in $C^0\cap W^{1,2}$ and Corollary \ref{8} we can shrink $I^t$ so that $\r(s)$ has energy at most $\frac{\e_1}{3}$ for $s\in 2I^t$ and, in adition,
\begin{equation}
    \begin{split}
       \Big| E(\r(s))-&E(\r(s),\frac{1}{4}\B_s)\\
       &E(\r(t))-E(\r(t),\frac{1}{4}\B_t)
       \Big|\leq\frac{e_{\frac{1}{4}\e_1,\r}}{4}. 
    \end{split}
\end{equation}
Since $I$ is compact, we can cover $I$ by finitely many $I^t$, say $\{I^{t_j},j=1,...,m\}$. Moreover, after discarding some of the intervals we can rearrange that each $t\in[0,1]$ is in at least one interval $\bar{I}^{t_j}$, and each $\bar{I}^{t_j}$ intersects at most two other $\bar{I}^{t_{j-1}}$ $\bar{I}^{t_{j+1}}$, and $I^{t_{j-1}}$ and $I^{t_{j+1}}$
doesn't intersect each other. Now for each $j=1,...,m$, we choose a continuous function $\mu_j:[0,1]\to[0,1]$ so that $\mu_j=1$ on $I^{t_j}$ and $\mu_j=0$ for $t\notin 2I^{t_j}\cap(I^{t_{j-1}}\cup I^{t_j}\cup I^{t_{j+1}})$.

The Lemma is proved from the estimates in the proof.
\end{proof}

\begin{lemma}\label{deformation lemma}
There is a constant $\e_0>0$ and a continuous function $\Psi:[0,\infty)\to[0,\infty)$ with $\Psi(0)=0$, such that for any $(\tilde{\r}(t),\s(t))\in\tilde{\O}$ with 
\[\max_{t\in[0,1]}E(\tilde{\r}(t),\s(t))\geq W_E(\O_{(\tilde{r},\s)})/2.\]
If $(\tilde{\r},\s)$ isn't harmonic unless $(\tilde{\r},\s)$ is a constant map, then we can perturb $(\tilde{\r},\s)$ to  $(\r(t),\s(t))\in\O_{(\tilde{\r},\s)}$ such that $E(\r(t),\s(t))\leq E(\tilde{\r}(t),\s(t))$.  Moreover, for any finite collection $\B$ of disjoint geodesic balls on $\S_{\s(t)}$ with the geodesic radius of each ball bounded above by $r_0$ (see \eqref{rho}) and the injective radius of the center of that ball on $\S^d_{\s(t)}$, and 
\[\int_{\cup_{\B}B}|\nabla\r|^2\leq\e_0,\] 
we have the following property:
\begin{enumerate}
    \item $B\in\B$ is an interior geodesic ball on $\S_{\s(t)}$, let $v$ be an energy minimizing map with the same boundary value as $\r(t)$ on $\frac{1}{64}B$. 
    \item $B\in\B$ is a geodesic ball which intersects $\l\S_{\s(t)}$ orthogonally, say $B=D_r^+$, let $v$ be a free boundary energy minimizing map with $v=\r$ on $\l_{r/64}^A$ and $v\subset N$ on $\l_{r/64}^C$.
\end{enumerate}
Then 
\begin{equation}\label{35}
    \int_{\frac{1}{64}\cup_{\B}B}|\nabla\r(t)-\nabla v|^2\leq\Psi(E(\tilde{\r}(t),\s(t))-E(\r(t),\s(t))).
\end{equation}
\end{lemma}
\begin{proof}
The perturbation from $(\tilde{\r}(t),\s(t))$ to $(\r(t),\s(t))$ is done by successive harmonic replacements on the collection of balls given by Lemma \ref{5.18}. Denote $\tilde{\r}(t)=\r^0(t)$ and $r^k(t)=H(\r^{k-1}(t),\mu_k(t)\B_k)$ for $k=1,..,m.$ Then $\r(t)=\r^m(t)$. $\r$ is homotopic to $\tilde{\r}$ since continuously shrinking the geodesic balls on which we make harmonic replacement give an explicit homotopy. Thus, $(\r,\s)\in\tilde{\O}_{(\tilde{\r},\s)}$.

Fixed $t\in(0,1)$ such that $E(\r(t),\s(t))\geq\frac{1}{2}W_E(\tilde{\O}_{(\tilde{\r},\s)})$, Lemma \ref{5.18} implies that for some $\mu_j(t)$ that the harmonic replacement for $\tilde{\r}(t)$ on $\frac{1}{4}\mu_j(t)\B_j\subset\S_{\s(t)}$ decreases the energy by at least $\frac{1}{8}e_{\frac{1}{8}\e_1,\r(t)}$. Thus, even in the case such that $\frac{1}{4}\mu_j(t)\B_j\subset\S_{\s(t)}$ is the second family of geodesic balls that we do free boundary harmonic replacement on. \eqref{24} of Lemma \ref{5.13} gives that
\[E(\tilde{\r}(t),\s(t))-E(\r(t),\s(t))\geq k\big(\frac{1}{8}e_{\frac{1}{8}\e_1,\tilde{\r}(t)}\big)^2.\]
We deform $\tilde{\r}(t)$ to $\r(t)$ by at most two free boundary harmonic replacement, we denote the first one by $\r^1(t)$. For any collection $\B\subset\S_{\s(t)}$ with 
\[\frac{1}{2}\int_{\B}|\nabla\r(t)|^2\leq\frac{1}{12}\e_1,\]
we can assume that both $\tilde{\r}(t)$ and $\r^1(t)$ have energy less than $\frac{1}{8}\e_1$ on $\B$.
Now using \eqref{25} twice for $\rho=\frac{1}{64}$, $\frac{1}{16}$, we have
\begin{equation}
    \begin{split}
        E&(\r(t))-E(H(\r(t),\frac{1}{64}\B))\\
        &\leq E(\r^1(t))-E(H(\r^1(t),\frac{1}{16}\B))+\frac{1}{k}\{E(\r^1(t))-E(\r(t))\}^{\frac{1}{2}}\\
        &\leq E(\tilde{\r}(t))-E(H(\tilde{\r}(t),\frac{1}{4}\B))+\frac{1}{k}\{E(\tilde{\r}(t))-E(\r^1(t))\}^{\frac{1}{2}}\\
        &+\frac{1}{k}\{E(\tilde{\r}(t))-E(\r(t))\}^{\frac{1}{2}}\\
        &\leq\frac{1}{8}e_{\frac{1}{8}\e_1,\tilde{\r}(t)}+C(E(\tilde{\r}(t))-E(\r(t)))^{\frac{1}{2}}\\
        &\leq C'(E(\tilde{\r}(t))-E(\r(t)))^{\frac{1}{2}}.
    \end{split}
\end{equation}
We finish the proof by taking $\e_0=\frac{1}{12}\e_1$ and $\Psi$ the square root function.
\end{proof}
\section{Convergence results}
\subsection{Degeneration of the conformal structure}
We've defined bordered Riemann surface back in section \ref{bordered riemann surface}, and mentioned that the fixed bordered Riemann surface $\S_0$ that we use as domain for the sweepout is a bordered Riemann surface with genus $g\geq 1$ and $m\geq 1$ ideal boundary components, and none of the ideal boundary component is a puncture. Moreover, the double $\S^d_0$ of such bordered Riemann surface carries a hyperbolic structure. For any bordered Riemann surface $\S$, there exists an anti-holomorphic involution map $\sigma:\S^d\to\S^d$ such that $\S=\S^d/\sigma$ and $\l\S$ is just the fixed point set of $\sigma$. Now we introduce \textbf{nodal bordered Riemann surface}: a smooth bordered Riemann $\S$ is said to be of type $(g,m)$ if $\S$ is topologically a sphere attached with $g$ handles and $m$ disks removed. It is topologically equivalent to a compact surface of genus $g$ with $m$ punctures. A nodal bordered Riemann surface is of type $(g,m)$ if it is a degeneration of a smooth bordered Riemann surface of the same type.

There are three types of nodes for a nodal bordered Riemann surface. Let $(z,w)$ be the coordinate on $\C^2$ and $\sigma(z,w)=(\bar{z},\bar{w})$ be the complex conjugation.  A node on a bordered Riemann surface is a singularity which is locally isomorphic to one of the following: \begin{enumerate}
    \item Interior node: $(0,0)\in\{zw=0\}$
    \item Type I boundary node: $(0,0)\in\{z^2-w^2=0\}/\sigma$
    \item Type II boundary node: $(0,0)\in\{z^2+w^2=0\}/\sigma$
\end{enumerate}
\begin{description}
\item[1]{Interior node}\\
Let $U$ be a neighborhood of $(0,0)\in\C^2$ and $\Z=\{(z,w)\in U|zw=0\}$. Then $\Z$ has two components $\Z_1=\{(z,w)\in U|z=w\}$ and $\Z_2=\{(z,w)\in U|z=-w\}$, which are attached at the nodal point $\Z_1\cap\Z_2=\{(0,0)\}$. Note that the metric near the node on each component is flat. 
\item[2]{Type I boundary node}\\
Let $\Z=\{(z,w)\in U|z^2-w^2=0\}$. Then has two components 
\[\Z_1=\{(z,w)\in U|z=w\}\text{ and } \Z_2=\{(z,w)\in U|z=-w\},\]
which are attached at the point $\Z_1\cap\Z_2=\{(0,0)\}$. The map $\sigma$ gives an involution on each component $\Z_i$. The fixed point set of the involution is $F=\{(z,w)\in\C^2|Im(z)=Im(w)=0\}$. The boundary of $\Z_1/\sigma$ is just the fixed point set of the involution in $\Z_1$, i.e., $c_1=F\cap\Z_1=\{z=w\in\R^1\}$, which is a 1-dimensional curve. Similarly, the boundary of $\Z_2/\sigma$ is the curve $c_2=F\cap\Z_2=\{z=-w\in\R^1\}$. Thus the boundary curves $c_1$ and $c_2$ intersects at the node $(0,0)$.

By sending. $(z,w)$ to $(z-w,z+w),$ it's clear
that the surface $\Z$ here is isomorphic to the one in the interior node case. $\Z$ is the complex double of $\Z/\sigma$. Thus by using the involution map, the picture of Type I boundary node is actually a vertical half of the interior node case.
\item[3]{Type II boundary node}\\
Let $\Z=\{(z,w)\in U|z^2+w^2=0\}$. Then $\Z$ has two components 
\[\Z_1=\{(z,w)\in U|z=iw\}\text{ and }\Z_2=\{(z,w)\in U|z=-iw\},\]
which are attached at the point $\Z_1\cap\Z_2=\{(0,0)\}$. The map $\sigma$ gives an anti-holomorphic involution between $\Z_1$ and $\Z_2$. The only fixed point of $\sigma$ on $\Z$ is the node $(0,0).$ Thus the boundary of $\Z/\sigma$ is simply the node itself.

By sending $(z,w)$ to $(z-iw,z+iw)$, it's clear that the surface $\Z$ here is also isomorphic to the one in the interior case. Indeed, $\Z$ is the complex double of $\Z/\sigma$. Thus by using the symmetry map $\sigma$, the picture of Type II boundary node is actually a horizontal half of the interior node case. 
\end{description}
Suppose $\S$ is a bordered surface of type $(g,m)$ with $g\geq 1$ and $m\geq 1$. Let $\S^d$ be the corresponding complex double which is a closed Riemann surface with genus $g^d=2g+m-1
\geq 2$. Then by uniformization theorem, for each complex sturcture $j$ on $\S$, there exists a hyperbolic metric $h$ on $\S^d$. It's easy to see that $h$ and $j$ can be extended to its complex double $\S^d$ such that $h^d$ and $j^d$ are symmetric w.r.t. the anti-holomorphic involution $\sigma.$ Now we denote the bordered hyperbolic surface by the triple $(\S,h,j)$ and state the following convergence result.
\begin{definition}\label{convergence}
Given a sequence  $\{(\S_n,h_n,j_n)\}_{n\in\N}$ of bordered Riemann surface with genus $g\geq 1$ and boundary component $m\geq 1$ such that its double carries a hyperbolic structure. $\{(\S_n,h_n,j_n)\}_{n\in\N}$ is said to converge to a nodal bordered Riemann surface $(\S_\infty,h_\infty,j_\infty)$ if there exists a sequence of finite sets
\[\L_n=\{\a^i_n,\b^j_n,c^k_n\}\subset\S_n\]
consisting of pairwise disjoint geodesic in $\S^d_n$, where
\begin{enumerate}
    \item $\a_n^i$ belongs to the boundary component of $\S_n$, i.e., $\a\in\l\S_n$.
    \item 
    $\b^j_n$ is an interior geodesic with its boundary lies on $\l\S_n$. 
    \item $c^k_n$ is an interior geodsesic of $\S_n$. 
\end{enumerate}
Such that there exists a sequence of continuous mapping 
\[\phi_n:\S_n\to\S_\infty,\]
satisfying the following conditions as $n\to\infty$:
\begin{enumerate}
\item $\phi_n(\a_n^i)=p^i$, where $p^i$ is a type I node of $\S_\infty$,
\item $\phi_n(\b_n^j)=q^j$, where $q^j$ is a type II node of $\S_\infty$,
\item $\phi_n(c_n^k)=r^k$, where $r^k$ is an interior node of $\S_\infty$,
\item 
\[\phi_n:\S_n\setminus\L_n\to\S_\infty\setminus\{p^i,q^j,r^k\},\]
is a diffeomorphism.
\item $\phi_n(h_n)\to h_\infty$ in $C^\infty_{loc}(\S_\infty\setminus\{p^i,q^j,r^k\}),$
\item $\phi_n(j_n)\to j_\infty$ in $C^\infty_{loc}(\S_\infty\setminus\{p^i,q^j,r^k\}).$
\end{enumerate}
\end{definition}
\begin{proposition}\cite[Proposition 5.1]{Gromov}\label{convergence result}
For any sequence of bordered Riemann surface $\{(\S_n,h_n,j_n)\}_{n\in\N}$, there exists a nodal bordered Riemann surface $(\S_\infty,h_\infty,j_\infty)$ such that $\{(\S_n,h_n,j_n)\}_{n\in\N}$ converges to up to subsequence.
\end{proposition}

\subsection{Hyperbolic geometry near nodes}

We have the following description of the geometry near the geodesic $c$ on a Riemann surface with hyperbolic structure. We refer to \cite[Chap. 4]{Gromov} for more details.
\begin{lemma}\cite[Chap.4 Lemma 1.6]{Gromov} \label{collar nbd}
For any simply connected geodesic $c$ with length $l$ in a hyperbolic surface $\S^d$, there exists a collar neighborhood of $c$, which is isomorphic to the following collar region in the hyperbolic plane $\H$:
\begin{equation}\label{collar}
\mathcal{C}(c)=\Big\{z=re^{i\theta}\in\H:1\leq r\leq e^l, \theta_0(l)
\leq\theta\leq\pi-\theta_0(l)\Big\},    
\end{equation}
with the circles $\{r=1\}$ and $\{r=e^l\}$ identified by the isometry.  $\G_l:z\to e^lz$. Here $\theta_0(l)=\tan^{-1}(\sinh(\frac{l}{2}))$ and the geodesic $c$ is isometric to $\{z=re^{i\frac{\pi}{2}}\in i\R:1\leq r\leq e^l\}.$
\end{lemma}

We can give a explicit metric on the collar region by a conformal change of coordinates. Now we can view the parameters $r$ and $\theta$ in \eqref{collar} as azimuthal and vertical coordinates for a cylinder respectively. Under the following transformation:
\[re^{i\theta}\to(t,\phi)=\big(\frac{2\pi}{l}\theta,\frac{2\pi}{l}\log(r)\big),\]
where $l$ is the length of the center, the collar region $\mathcal{C}(c)$ is conformally transformed to a cylinder
\[\P(c)=\Big\{(t,\phi):\frac{2\pi}{l}\theta_0\leq t\leq\frac{2\pi}{l}(\pi-\theta_0),0\leq\phi\leq 2\pi\Big\},\]
and the hyperbolic metric $ds^{2}_{-1}=\frac{dz^2}{|Im(z)|^2}$ is expressed as 
\[ds^{2}_{-1}=\Big(\frac{l}{2\pi\sin(\frac{l}{2\pi}t)}\Big)^2(dt^2+d\phi^2),\]
which is conformal to the standard cylindrical metric $ds^2=dt^2+d\phi^2$. We can see that if the geodesic $c$ shrink to a point, then $l\to 0$, and a conformally infinitely long cylinder will appear.

Suppose $\{(\S_n,h_n,j_n)\}_{n\in\N}$ is a sequence of degenerating hyperbolic surfaces which converges to a hyperbolic surface $(\S_{\infty},h,j)$ with finitely many nodes $\{x_1,...,x_k;z_1,...,z_l\}$ where $x_i$'s are on the boundary and $z_i$'s are interior nodes. Then the complex $\S^d_n$ double of $\S_n$ converges to the complex double $\S^d_{\infty}$ of $\S_{\infty}$ with nodes $\{x_1,...,x_k;z_1,...,z_l,z'_1,...,z_l'\}$ where $z'_j=\sigma(z_j)\in\S^d_{\infty}$. $\S_{\infty}^d$ is obtained by pinching $k+2l$ pairwise disconnected Jordan curves to the nodes.
 
For simplicity, we first assume there is only one node $z$. If $z$ is an interior nodes, i.e. $z\neq\l\S_\infty$, then there exists a sequence of collar area $\mathcal{C}(c_n)\subset\S_n$ near a sequence of closed geodesics $c_n\subset\S_n$, which is isomorphic to a sequence of hyperbolic cylinder $\P(c_n)=[-T_n,T_n]\times S^1$ with $T_n\to\infty$. Moreover, $\S_n\setminus\mathcal{C}(c_n)$ converges to $\S_\infty\setminus\{z\}$ smoothly. The local geometry near the node $z$ on each component of $\S_\infty\setminus\{z\}$ is a standard hyperbolic cusp.

If $x\in\l\S_{\infty}$ is a boundary node, then there also exists a collar area $\mathcal{C}(c_n)$ which lies in the double $\S^d_{\infty}$, near a closed geodesic $c_n\subset\S^d_n$ and isomorphic to a hyperbolic cylinder $\tilde{P}_n=[-T_n,T_n]\times S^1$. Note that the cylinder $\tilde{P}_n$ is symmetric w.r.t. the involution $\sigma$, i.e., $\sigma^*ds^2=ds^2$. However, there are only two possible involutions on the hyperbolic cylinder $\tilde{P}_n$ which is anti-holomorphic as follows
\begin{enumerate}
\item
The first one corresponds to the symmetry $\tilde{P}_n$ w.r.t. the horizontal lines 
\[\xi^1_n=\{(t,\theta)|\theta=0,\pi\}\subset\tilde{P}_n.\]
Namely, the involution $\sigma:\tilde{P}_n\to\tilde{P}_n$ maps $(t,\theta)$ to $(t,2\pi-\theta)$. In this case, $\xi^1_n$ lies in the boundary $\l\S_n$. Let $Q^+_n=[-T_n,T_n]\times[0,\pi]$ be the vertical half of $\tilde{P}_n$ in $\S_n$. Then $\S_n\setminus Q^+_n$ converges to $\S_{\infty}\setminus\{x\}$ smoothly. The node is obtained by shrinking geodesic $c_n$ which connects two points at the boundary $\l\S_n$ to the point $x$. This is exactly the type I node described in the previous section.

If we adapt the collar neighborhood $\mathcal{C}(c_n)$ defined in Lemma \ref{collar nbd} and assume that the geodesic $c_n$ is isometric to $\{z=re^{i\frac{\pi}{2}}\in i\R:1\leq r\leq e^l\}.$ Then we can write the \textit{half cylinder} that degenerates to type I node as the following
\begin{equation}\label{type1}
\P^1(c_n)=\Big\{(t,\phi):\frac{2\pi}{l}\theta_0(l)\leq t\leq\frac{2\pi}{l}(\pi-\theta_0(l)),0\leq\phi\leq \pi\Big\},    
\end{equation}
where $\theta_0(l)=\tan^{-1}(\sinh(\frac{l}{2}))$, and $l\to 0$ as $c_n$ shrinks to a point.
\item
The second one corresponds to the symmetry of $\tilde{P}_n$ w.r.t. the vertical middle circle
\[\xi^2_n=\{(t,\theta)|t=0\}\subset\tilde{P}_n.\]
Namely, involution $\sigma:\tilde{P}_n\to\tilde{P}_n$ maps $(t,\theta)\to(-t,\theta)$. In this case, $\xi^2_n$ is fixed by $\sigma$ and hence belongs to the boundary $\l\S_n$. Let $\tilde{P}^+_n=[0,T_n]\times S^1$ be the horizontal half of $\tilde{P}_n$ in $\S_n$. Then $\S_n\setminus\tilde{P}^+_n$ converges to $\S_\infty\setminus\{x\}$ smoothly. The node is obtained by shrinking a boundary curve (also can be viewed as a geodesic in the double $\S^d_n$) to a point. This is exactly type II node described in the previous section. 

If we adapt the collar neighborhood $\mathcal{C}(c_n)$ defined in Lemma \ref{collar nbd} and assume that the geodesic $c_n$ is isometric to $\{z=re^{i\frac{\pi}{2}}\in i\R:1\leq r\leq e^l\}.$ Then we can write the \textit{half cylinder} that degenerates to type II node as the following
\begin{equation}\label{type2}
\P^2(c_n)=\Big\{(t,\phi):\frac{2\pi}{l}\theta_0(l)\leq t\leq\frac{2\pi}{l}(\frac{\pi}{2}),0\leq\phi\leq 2\pi\Big\},    
\end{equation}
where $\theta_0(l)=\tan^{-1}(\sinh(\frac{l}{2}))$, and $l\to 0$ as $c_n$ shrinks to a point. Note that $\frac{2\pi}{l}(\frac{\pi}{2})$ is $\frac{1}{2}\Big(\frac{2\pi}{l}\theta_0(l)+\frac{2\pi}{l}(\pi-\theta_0(l))\Big)$.

\end{enumerate}
\subsection{Convergence}
\subsubsection{Convergence on domains}
In the section we consider the following sequence: $\{\r_n\}_{n\in\N}$ such that
\begin{enumerate}
    \item $\r_n\in C^0(\bar{\S}_n,M)\cap W^{1,2}(\S_n,N)$,
    \item $\r_n|_{\l\S_n}\subset N$,
    \item $\text{Area}(\r_n)-E(\r_n,\S_n)\to 0$
    \item For any finite collection $\B$ of disjoint geodesic balls on $\S_n$ with radius less than $r_0$ (see \eqref{rho}) and the injective radius of the center of that ball on $\S_n^d$. If the energy of $r_n$ on $\B$ is less than $\e_0$ as Lemma \ref{deformation lemma}. Let $B_i\in\B$ be the interior geodesic ball, and let $v$ be the harmonic replacement map  with $v=\r_n$ on $\cup_i\frac{1}{64}B_i$, then we have
    \begin{equation}\label{38}
    \int_{\cup_i\frac{1}{64}B_i}|\nabla\r_n-\nabla v|^2\to 0,    
    \end{equation}
    \item For any finite collection $\B$ of disjoint geodesic balls on $\S_n$ with radius less than $r_0$ (see \eqref{rho}) and the injective radius of the center of that ball on $\S_n^d$. If the energy of $r_n$ on $\B$ is less than $\e_0$ as Lemma \ref{deformation lemma}. Let $D_i^+\in\B$ be the geodesic ball such that it intersects the boundary $\l\S_n$ orthogonally, let $v$ be the free boundary harmonic replacement map with $v=\r_n$ on $\frac{1}{64}\l^A_i$ and $v\subset N$ on $\frac{1}{64}\l^C_i$, then we have
    \begin{equation}\label{39}
    \int_{\cup_i\frac{1}{64}D^+_i}|\nabla\r_n-\nabla v|^2\to 0.    
    \end{equation}
\end{enumerate}
Since area and energy are both conformal invariant, we can choose representatives $(\S_n,h_n,j_n)$ in the conformal classes of $\{\S_n\}_{n\in\N}\subset\T(\S_0)$. We say that $\{\S_n\}_{n\in\N}$ converges to $\S_\infty$ (possibly a nodal bordered Riemann surface) if we can find hyperbolic representatives $(\S_n,h_n,j_n)\in\S_n$ and $(\S_\infty,h_\infty,j_\infty)\in\S_\infty$ such that $(\S_n,h_n,j_n)$ converges to $(\S_\infty,h_\infty,j_\infty)$ in the sense of Definition \ref{convergence}. By Proposition \ref{convergence result}, given $\{\S_n\}_{n\in\N}$, there exists a nodal bordered Riemann surface $(\S_\infty,h_\infty,j_\infty)$ such that $(\S_n,h_n,j_n)\in\S_n$ converges to up to subsequence in the sense of Definition \ref{convergence}. By Definition \ref{convergence}, we know that that there exists a sequence of continuous mapping such that 
\[\phi_n:\S_n\to\S_\infty,\]
satisfying the following conditions as $n\to\infty$:
\begin{enumerate}
\item $\phi_n(\a_n^i)=p^i$, where $p^i$ is a type I node of $\S_\infty$,
\item $\phi_n(\b_n^j)=q^j$, where $q^j$ is a type II node of $\S_\infty$,
\item $\phi_n(c_n^k)=r^k$, where $r^k$ is an interior node of $\S_\infty$,
\item 
\[\phi_n:\S_n\setminus\L_n\to\S_\infty\setminus\{p^i,q^j,r^k\},\]
is a diffeomorphism.
\item $\phi_n(h_n)\to h_\infty$ in $C^\infty_{loc}(\S_\infty\setminus\{p^i,q^j,r^k\}),$
\item $\phi_n(j_n)\to j_\infty$ in $C^\infty_{loc}(\S_\infty\setminus\{p^i,q^j,r^k\}).$
\end{enumerate}
We denote the set $\S_\infty\setminus\{p^i,q^j,r^k\}$ by $\S_\infty^*$.
\begin{lemma}\label{domain convergence}
Given a sequence $\{\r_n\}_{n\in\N}$ as above with bounded energy. Then there exists finitely many points $\{x_1,...,x_l\}\subset\S_\infty^*$ and a free boundary conformal harmonic map $u_0$ such that
\begin{enumerate}
    \item $u_0\in C^0(\bar{\S}^*_\infty,M)\cap W^{1,2}(\S^*_\infty,M)$,
    \item $u_0|_{\l\S^*_{\infty}}\subset N$,
    \item for any compact subset $K\subset\S_\infty^*\setminus\{x_1,...,x_l\}$, we have $\r_n\circ\phi_n^{-1}$ converges to $u_0$ in $W^{1,2}$ on $K$ up to subsequence.
\end{enumerate}
\end{lemma}
The proof of Lemma \ref{domain convergence} is similar to \cite{SU}\cite{CD}\cite[Lemma 5.1]{minimaltorus}\cite[Lemma 7.2]{LAX}\cite{PR}, so we omit the proof here. 

We've discussed at the beginning of Chapter 5 that for $\r_n\in C^0(\S_n,M)\cap W^{1,2}(\S_n,M)$, $\S_n\in\T(\S_0)$, since the double of $\S_n$ carries a hyperbolic structure and Claim \ref{invariant of boundary} implies that $\pi^{-1}_{\s_n}(\S_n)$ is fixed $\forall n\in\N$, we can pull $\r_n$ back to a fixed subset $\H'\subset\H$, and we can assume without loss of generality that one of the components of $\l\S_n$ can be pulled back to $i\R\subset\H.$ Let $\psi_b$ be defined as the following
\[\psi_b=\frac{z-(b+1)}{z-(b-1)},\quad b\in i\R\subset\H.\]
$\psi_b$ is a conformal map that sends $\{(x,y)\in\C^2,\:x>0\}$ to the unit disk and $b\to(-1,0)$.
\begin{definition}\label{bubble tree convergence}
Given a sequence of maps $\{\r_n\}_{n\in\N}$, $\r_n:\S_n\to M$ and 
\begin{enumerate}
    \item $u_0:\S_\infty\to M$, $u_0$ is a conformal harmonic map with free boundary on $N$, i.e., $u_0|_{\l\S_\infty}\subset N$,
    \item $v_1,...,v_r:S^2\to M$ non constant minimal spheres,
    \item $\w_1,...,\w_s:D\to M$ non constant minimal disk with free boundary on $N$.
\end{enumerate}
We say that the sequence $\{\r_n\}_{n\in\N}$ converges to $\{u_0,v_1,...,v_r,\w_1,...,\w_s\}$ in bubble tree sense if the following holds
\begin{enumerate}
    \item there exists a finite set $\{x_1,...,x_l\}\in\S_\infty^*$ such that $\r_n\circ\phi^{-1}_n$ converges to $u_0$ in $W^{1,2}_{loc}(\S_\infty^*\setminus\{x_1,...,x_l\})$;
    \item there exists $a_1^n,...,a_r^n\in\S_n$ and scales $\lambda_1^n,...,\lambda_r^n\to 0$ such that 
    $\r_n(a^n_i+\lambda_i^n\cdot)$ converges to $v_i\circ\Pi_S^{-1}$ in $W^{1,2}_{loc}(\R^2\setminus\F_i)$ for $1\leq i\leq r$, where $\Pi_S:S^2\to\R^2$ is the stereographic projection with respect to the north pole and $\F_i$ is a finite set;
    \item there exists $b_1^n,...,b_s^n\in\l\S_n$ and scales $\zeta_1^n,...,\zeta_s^n\to 0$ such that $\r_n\circ (\zeta_j^n\psi^{-1}_{b_j^n}(\cdot))$ converges to $\w_j$ in $W^{1,2}_{loc}(D\setminus\F_j)$ for $1\leq j\leq s$, where $\F_j$ is a finite set.
    
\end{enumerate}
\end{definition}
\begin{definition}
Let $\{\r_n\}_{n\in\N}$ be given as above. For $x\in\S_\infty$ and $\phi_n^{-1}(x)\in\S_n$, we call such $x$ an interior energy concentration point of the sequence $\{\r_n\}_{n\in\N}$ if 
    \[\inf_{r>0}\Big[\int_{B_r(\phi_n^{-1}(x))}|\nabla\r_n|^2\Big]\geq\e_{su}\]
    where $\e_{su}$ is the constant given in \cite[Main Estimate 3.2]{SU}. 
    
    Similarly, for $x\in\bar{\S}_\infty$ and $\phi_n^{-1}(x)\in\l\S_n$, we call such $x$ a boundary energy concentration point of the sequence $\{\r_n\}_{n\in\N}$ if 
    \[\inf_{r>0}\Big[\:\int_{B_r(\phi_n^{-1}(x))}|\nabla\r_n|^2\Big]\geq\e_{F}\]
    where $\e_{F}$ is the constant given in \cite[Proposition 1.7]{AF}.
\end{definition}
\begin{lemma}\label{degenerating convergence}
Given a sequence $\{\r_n\}_{n\in\N}$ as above with bounded energy. There exists 
\begin{enumerate}
    \item $u_0:\S_\infty\to M$, $u_0$ is a conformal harmonic map with free boundary on $N$, i.e., $u_0|_{\l\S_\infty}\subset N$,
    \item $v_1,...,v_r:S^2\to M$ non constant minimal spheres,
    \item $\w_1,...,\w_s:D\to M$ non constant minimal disk with free boundary on $N$,
\end{enumerate}
such that $\{\r_n\}_{n\in\N}$ converges to in bubble tree sense up to subsequence. Moreover, we have the following energy identity:
\begin{equation}\label{energy identity}
\lim_{n\to\infty}E(\r_n,\S_n)=E(u_0,\S_\infty)+\sum_i E(v_i,S^2)+\sum_j E(\w_j,D).    
\end{equation}
\end{lemma}
\begin{proof}
Lemma \ref{domain convergence} implies that there exists a finite set $\{x_1,...,x_l\}\in\S_\infty^*$ such that $\r_n\circ\phi^{-1}_n$ converges to $u_0$ in $W^{1,2}_{loc}(\S_\infty^*\setminus\{x_1,...,x_l\})$.

\cite[Theorem 5.6]{minmaxgenus} implies that for each interior energy concentration point $x$ there exist $a^n\in\S_n$ such that $\phi_n(a^n)\to x$ and scales $\lambda^n\to 0$ such that 
    $\r_n(a^n+\lambda^n\cdot)$ converges to a minimal sphere $v\circ\Pi_S^{-1}$ in $W^{1,2}_{loc}(\R^2\setminus\F)$, where $\Pi_S:S^2\to\R^2$ is the stereographic projection with respect to the north pole and $\F$ is a finite set. 
    
So we are left to show the bubble  convergence for the boundary energy concentration point. That is (3) of Definition \ref{bubble tree convergence} and prove the energy identity of it. 

For a boundary energy concentration point $x\in\bar{\S}_\infty$, $\phi_n^{-1}(x)\in\l\S_n$. Let $x_n:=\phi_n^{-1}(x)$. For each $n$, we choose $\p_n$ to be the smallest radius such that
\[\int_{B_{\p_n}(x_n)}|\nabla\r_n|^2=\e_F,\]
where $\e_F$ is the constant given in \cite[Proposition 1.7]{AF}. Since $x$ is a boundary energy concentration point,     $\inf_{r>0}\Big[\:\int_{B_r(x_n)}|\nabla\r_n|^2\Big]\geq\e_{F}$ implies that $\p_n\to 0$.

$B_{\p_n}(x_n)$ is a geodesic ball centered at $y_n\in\l\S_n$ which intersects $\l\S_n$ orthogonally. We can pull the geodesic ball $B_{\p_n}(x_n)$ back to the upper Poincar\'{e} disk $(D^+,ds^2_{-1})$ and denote it by $D^+_{n}$. Now we rescale $D^+_{n}$ to $\{z\in\C,\|z\|\leq 1\}\cap\H$ by $x\to x/\p^0_n$, where $\p^0_n$ is the Euclidean radius of $D^+_{n}$ measured w.r.t. the Euclidean metric on $(D^+,ds^2_0)$. The rescaled hyperbolic metric $ds_{-1}^2$ is
\[ds^2_n=\frac{dz^2}{1-|\p^0_nz|^2},\]
since $\p^0_n\to 0$, $ds^2_n$ converges to flat metric on any compact set on $\H$. Now we consider the following sequence 
\[\{(\r_n(\p_n^0\cdot),(D^+_n,ds^2_n))\}_{n\in\N}.\]
Since \eqref{38} and \eqref{39} are scaling invariant, the sequence $\{(\r_n(\p_n^0\cdot),(D^+_n,ds^2_n))\}_{n\in\N}$ satisfies the assumption of Lemma \ref{domain convergence} so $\{(\r_n(\p_n^0\cdot),(D^+_n,ds^2_n))\}_{n\in\N}$ converges to a harmonic map $\w$ defined on $\H$ in the sense of Lemma \ref{domain convergence}. We can consider $w$ as a harmonic map defined on unit disk $D$ with free boundary $w|_{\l D}\subset N$ by considering the sequence $\r_n(\p_n^0\psi_{x_n}^{-1}(\cdot))$ instead. 

In order to prove \eqref{energy identity}, we need to study the behavior of the limit process on some \textit{half cylinder} neighborhoods of degenerating geodesics. Given a boundary energy concentration point $x$. Let $B_r(x_n)$ be the geodesic ball which centered at $\phi_n^{-1}(x)\in\l\S_n$, with radius $r$ w.r.t. the hyperbolic metric and it intersects $\l\S_n$ orthogonally. We can pull $B_r(x_n)$ back to the upper Poincar\'{e} disk $(D^+,ds^2_{-1})$, centered at the origin, and denote it by $D^+(x_n,r)$. Near a boundary energy concentration point $x_n$, if we compare the energy limit $E(\r_n,D^+(x_n,r))$ with the sum of $E(u_0,D^+(x_n,r))$ and $E(\r_n(\p_n^0\cdot),D^+(x_n,r/\p_n^0))$, we need to count the \textit{neck} part, which is given by 
\[\lim_{r\to 0,R\to\infty}\lim_{n\to\infty}E(\r_n,D^+(x_n,r)\setminus D^+(x_n,R\p_n^0)).\]
We will use $\mathcal{C}'_{a,b}$ to denote the half cylinder $[a,b]\times[0,\pi]$. Under the change of coordinates $(r,\theta)\to(t,\theta)=(\log(r),\theta)$, the half annuli $D^+(x_n,r)\setminus D^+(x_n,R\p_n^0)$ are changed to half cylinder $\mathcal{C}'_{r_n^1,r_n^2}$, where $r_n^1=\log(\p^0_nR)$ and $r_n^2=\log(r)$. The corresponding hyperbolic metric is \[ds^2_{-1}=\frac{e^{2t}}{1-e^{2t}}(dt^2+d\theta^2).\]
When we rescale the metric so that the center slice $\frac{1}{2}(r_n^1+r_n^2)=\frac{1}{2}(\log(\p^0_nR)+\log(r))$ has length $\pi$. Then the metric converges to the flat metric on any compact subset of the infinite long cylinder $\R\times[0,\pi]$.

Now let us see the behavior near degenerating boundary. We need to consider the \textit{half collar} neighborhood of the degenerating boundary. Let $c_n$ denote the degenerating curve; 
\begin{description}
\item[1] Type I node,\\
By \eqref{type1}, we can write the neighborhood of the type I node as the following:
\[\P^1(c_n)=\Big\{(t,\theta):\frac{2\pi}{l}\theta_0(l)\leq t\leq\frac{2\pi}{l}(\pi-\theta_0(l)),0\leq\theta\leq \pi\Big\},\]
where $\theta_0(l)=\tan^{-1}(\sinh(\frac{l}{2}))$, and $l\to 0$ as $c_n$ shrinks to a point. $\P^1(c_n)$ has the following hyperbolic metric: 
\[ds^{2}_{-1}=\Big(\frac{l}{2\pi\sin(\frac{l}{2\pi}t)}\Big)^2(dt^2+d\theta^2).\]
If we rescale the metric so that the center slice (which corresponds to the degenerating curve $c_n$) $\{(\frac{\pi^2}{l},\theta),0\leq\theta\leq\pi\}$ has length $\pi$, then the hyperbolic metric converges to the flat metric on any compact subset of $\R\times[0,\pi]$.
\item[2] Type II node,\\
By \eqref{type2}, we can write the neighborhood of the type II node as the following:
\[\P^2(c_n)=\Big\{(t,\theta):\frac{2\pi}{l}\theta_0(l)\leq t\leq\frac{2\pi}{l}(\frac{\pi}{2}),0\leq\theta\leq 2\pi\Big\},\]
where $\theta_0(l)=\tan^{-1}(\sinh(\frac{l}{2}))$, and $l\to 0$ as $c_n$ shrinks to a point. $\P^2(c_n)$ has the following hyperbolic metric: 
\[ds^{2}_{-1}=\Big(\frac{l}{2\pi\sin(\frac{l}{2\pi}t)}\Big)^2(dt^2+d\theta^2).\]
If we rescale the metric so that slice which corresponds to the degenerating curve $c_n$, $\{(\frac{\pi^2}{l},\theta),0\leq\theta\leq 2\pi\}$, has length $2\pi$, then the hyperbolic metric converges to the flat metric on any compact subset of $\R\times[0,2\pi]$.
\end{description}

Since type II node has the same infinitely long cylinder structure as the interior node, except that $\r_n$ sends one end $\{(\frac{\pi^2}{l},\theta),0\leq\theta\leq 2\pi\}$ to $N$. It's similar to the argument of \cite[Theorem 5.6]{minmaxgenus} so we omit the proof here.

We can see that type I node and the neck of energy concentration point have the same infinitely long half cylinder structure. We donate such half cylinder by $\mathcal{C}'_{r_n^1,r_n^2}$. Let $\e_2$ be the constant given in \cite[Theorem 7.8]{LAX}. We know that there exists a fixed constant $L>0$ such that 
\[E(\r_n,\mathcal{C}'_{t_n-L,t_n+L})\geq\e_2,\]
for some $t_n$, or else it contradicts \cite[Theorem 7.8, Lemma 7.9]{LAX}. That implies that the energy concentrate on some finite part of the cylinder. Now we repeat the bubble convergence procedure of $\r_n$ on $\mathcal{C}'_{r_n^1,r_n^2}$ so that the energy of $\r_n$ on $\mathcal{C}'_{t_n-L,t_n+L}$ is counted by some non-trivial conformal harmonic map (could be minimal spheres or minimal disk with free boundary on $N$) and then it decrease the energy of a definite amount. Then we get the energy of $\r_n$ on some half long cylinder $\mathcal{C}'_{a_n,b_n}$ is bounded by $\e_2>0$, which goes to $0$ as $n\to\infty$, or else it contradicts \cite[Theorem 7.8, Lemma 7.9]{LAX} as $\text{Area}(\r_n,\mathcal{C}'_{a_n,b_n})-E(\r_n,\mathcal{C}'_{a_n,b_n})\to 0.$
\end{proof}
\subsubsection*{Theorem 1.1}
\textit{
For any homotopically nontrivial sweepout $\r\in\O$. If $\W(\O_{\r})>0$, then there exist a bordered Riemann surface of same type $(g,m)$ (possibly a nodal bordered Riemann surface), conformal harmonic map $u_0:\S\to M$ with free boundary $u_0|_{\l\S}\subset N$, finitely many harmonic spheres $v_i:S^2\to M$, finitely many free boundary minimal disk $\w_j:D\to N$ with $\w_j|_{\l D}\subset N$, and a minimizing sequence $\r_n$ in $\O_{\r}$, $\{t_n\}\subset[0,1]$, such that $\r_n(t_n)$ converges to $\{u_0,v_i,\w_j\}$ in the bubble tree sense, and
\[\text{Area}(u_0)+\sum_{i,j}\text{Area}(v_i)+Area(\w_j)=\W(\O_{\r}).\]}
\begin{proof}{Proof of Theorem 1.1}
For a given sweepout $\r\in\O$, with a fixed borderd Riemann surface $\S_0$ of type $(g,m)$ as domain, where $g\geq 1$, $m\geq 1$. We consider $(\r,\s_0)\in\tilde{\O}_{\r,\s_0}$. By Corollary \ref{almost harmonic}, there exists a sequence $\{\r'_j,\s_j\}\subset\tilde{\O}_{(\r,\s_0)}$ such that 
\begin{equation}\label{43}
\lim_{j\to\infty}\max_{t\in[0,1]}\text{Area}(\r'_j(t))=\W(\O_{\r}),
\end{equation}
and 
\begin{equation}\label{44}
\lim_{j\to\infty}\text{Area}(\r'_j(t))-E(\r'_j(t),\s_j(t))=0.    
\end{equation}
By Lemma \ref{deformation lemma}, we can perturb the sequence $(\r_j',\s_j)$ to $(\r_j,\s_j)\subset\tilde{\O}_{(\r,\s_0)}$ so that $(\r_j,\s_j)$ satisfies \eqref{43}, \eqref{44}, and \eqref{35}. \eqref{35} implies that the $W^{1,2}$ difference between $\r_j$ and the free boundary harmonic replacement map on the geodesic balls on $\S_j$ which $\r_j$ has sufficiently small energy goes to $0$ as $j\to\infty$. That implies that the sequence $(\r_j,\s_j)$ satisfies the assumption of Lemma \ref{degenerating convergence}. Thus there exists 
\begin{enumerate}
    \item $u_0:\S_\infty\to M$,  where $\S_\infty$ is the limit of $\S_j$ in the sense of Definition \ref{convergence}. $u_0$ is a conformal harmonic map with free boundary on $N$, i.e., $u_0|_{\l\S_\infty}\subset N$,
    \item $v_1,...,v_r:S^2\to M$ non constant minimal spheres,
    \item $\w_1,...,\w_s:D\to M$ non constant minimal disk with free boundary on $N$,
\end{enumerate}
such that $\{(\r_j,\s_j)\}_{j\in\N}$ converges to in bubble tree sense up to subsequence. Moreover, we have the following energy identity:
\begin{equation}
\lim_{j\to\infty}E(\r_j,\s_j)=E(u_0,\S_\infty)+\sum_i E(v_i,S^2)+\sum_j E(\w_j,D).    
\end{equation}
Which implies 
\[\text{Area}(u_0)+\sum_{i,j}\text{Area}(v_i)+\text{Area}(\w_j)=\W(\O_{\r}).\]
\end{proof}
\bibliographystyle{plain}
\bibliography{ref.bib}	
\end{document}